\documentclass[a4paper,12pt]{article}
\usepackage[latin1]{inputenc}
\usepackage{amsthm}
\usepackage{amsfonts}
\usepackage[dvips]{graphicx}
\usepackage{fancyhdr}
\usepackage{latexsym}
\usepackage{amsmath}

\newtheorem{lem}{Lemma}
\newtheorem{thm}{Theorem}
\newtheorem{defn}{Definition}
\newtheorem{prop}{Proposition}
\newtheorem{cor}{Corollary}
\newtheorem{oss}{Remark}

\allowdisplaybreaks[4]

\begin{document}

\title{\large\textbf{GLOBAL AND TRAJECTORY ATTRACTORS FOR A
NONLOCAL CAHN-HILLIARD-NAVIER-STOKES SYSTEM}}

\author{
{\sc Sergio Frigeri}\\
Dipartimento di Matematica {\it F. Enriques}
\\Universit\`{a} degli Studi di Milano\\Milano I-20133, Italy\\
\textit{sergio.frigeri@unimi.it}
\\
\\
{\sc Maurizio Grasselli}\\
Dipartimento di Matematica {\it F. Brioschi}\\
Politecnico di Milano\\
Milano I-20133, Italy \\
\textit{maurizio.grasselli@polimi.it}}

\date{}

\maketitle

\begin{abstract}\noindent
The Cahn-Hilliard-Navier-Stokes system is based on a well-known
diffuse interface model and describes the evolution of an
incompressible isothermal mixture of binary fluids. A nonlocal variant
consists of the Navier-Stokes equations suitably coupled with a
nonlocal Cahn-Hilliard equation. The authors, jointly with P.~Colli,
have already proven the existence of a global weak solution to a
nonlocal Cahn\discretionary{-}{-}{-}Hilliard-Navier-Stokes system
subject to no-slip and no-flux boundary conditions.
Uniqueness is still an open issue even in dimension two.
However, in this case, the energy identity holds.
This property is exploited here to define, following J.M.~Ball's approach,
a generalized semiflow which has a global attractor. Through a similar
argument, we can also show the existence of a (connected) global attractor
for the convective nonlocal Cahn-Hilliard equation with a given
velocity field, even in dimension three.
Finally, we demonstrate that any weak solution fulfilling the energy
inequality also satisfies an energy inequality.
This allows us to establish the
existence of the trajectory attractor also in dimension three with
a time dependent external force.
\\ \\
\noindent \textbf{Keywords}: Navier-Stokes equations, nonlocal
Cahn-Hilliard equations, incompressible binary fluids, global
attractors, trajectory attractors.
\\
\\
\textbf{AMS Subject Classification}: 35Q30, 37L30, 45K05, 76T99.\end{abstract}

\section{Introduction}\setcounter{equation}{0}
Diffuse-interface methods in Fluid Mechanics are widely used by many
researchers in order to describe the behavior of complex fluids (see,
e.g., \cite{AMW,D} and references therein). A typical example is a
mixture of two incompressible fluids like, e.g., oil and water. To
describe the evolution of such a system a sufficiently simple model is
the so-called H model (see \cite{HH}, cf. also \cite{GPV,JV,M} and
references therein). This consists in a suitable coupling of the
Navier-Stokes equations for the (average) fluid velocity $u$, with a
Cahn-Hilliard type equation for the order parameter $\varphi$ (i.e.,
the relative concentration of one fluid or the difference of the two
concentrations). Temperature variations are neglected and the density
is supposed to be constant. This kind of system, called
Cahn-Hilliard-Navier-Stokes system, has been analyzed by several
authors both theoretically (see, for instance, \cite{A1,A3, B,
GG1,GG2,GG3,GP,S,ZWH}) and numerically (cf., e.g.,
\cite{BCB,B3,F,KSW,KKL,LS,SY}). Generalizations to unmatched
densities and compressible case have also been investigated (see
\cite{A2,AF,B2}). On the other hand, it is well know that the usual
Cahn-Hilliard equation can be viewed as a local approximation of a
nonlocal Cahn-Hilliard equation (see, for instance,
\cite{BH1,BH2,G,GZ,GL1,GL2,H,LP}). However, the corresponding
nonlocal version of the Cahn-Hilliard-Navier-Stokes system has been
analyzed only recently in \cite{CFG}. Nonetheless it is worth
mentioning that there exist some related works devoted to
liquid-vapor phase transitions (i.e., the so-called
Navier-Stokes-Korteweg systems) in which nonlocal energy
functionals are considered  (see \cite{R,R2}, cf. also \cite{Has}).

More precisely, we want consider the following system (see \cite{CFG} for details)
\begin{align}
& \varphi_t+u\cdot\nabla\varphi=\Delta\mu,\label{sy1}\\
&\mu=a\varphi-J\ast\varphi+F^\prime(\varphi), \label{sy2}\\
&u_t-\mbox{div}(\nu(\varphi)2Du)+(u\cdot\nabla)u+\nabla\pi=\mu\nabla\varphi+h(t),
\label{sy3}\\
&\mbox{div}(u)=0, \label{sy4}
\end{align}
in $\Omega\times (0,\infty)$, where $\Omega \subset
\mathbb{R}^d$, $d=2,3$, is a bounded domain with a sufficiently
smooth boundary and the density has been taken equal to one.
Here $J:\mathbb{R}^d \to \mathbb{R}$ is the
interaction kernel and
\begin{equation}
(J\ast\varphi)(x):=\int_{\Omega}J(x-y)\varphi(y)dy,\qquad a(x):=
\int_{\Omega}J(x-y)dy,\qquad x\in\Omega.
\label{defp}
\end{equation}
We recall that $F$ is the potential accounting for the
presence of two phases, while $\nu>0$ denotes the viscosity, $\pi$ the pressure,
$2Du:=\nabla u+(\nabla u)^{tr}$ and $h$ represents an
external force acting on the mixture.

In \cite{CFG}, jointly with P. Colli, we have proven the existence of a global weak solution
for system \eqref{sy1}-\eqref{sy4} endowed with the following
boundary and initial conditions
\begin{align}
&\frac{\partial\mu}{\partial n}=0,\quad u=0,\quad\mbox{on }
\partial\Omega\times (0,T),\label{nslip}\\
&u(0)=u_0,\quad\varphi(0)=\varphi_0,\quad\mbox{in }\Omega,
\label{sy6}
\end{align}
where $n$ is the unit outward normal to $\partial\Omega$.
This result has been obtained by
assuming that $F$ is sufficiently smooth and of arbitrary polynomial
growth. In addition, we have shown some regularity properties of the
solution provided that $F$ satisfies a reasonable coercivity condition.
In particular, such properties entail the validity of an energy identity in
dimension two. However, even in this case, uniqueness is still an open
issue. This is due to the lack of regularity of the order parameter
$\varphi$ which is a consequence of the presence of the nonlocal term
in place of the usual Laplace operator acting on $\varphi$ (see
\cite{CFG} for details). On the other hand, finding stronger solutions
does not seem straightforward as well. Thus, even in two dimensions,
the analysis of the (global) longtime behavior appears to be rather
challenging. Fortunately, at least in this case, we have an energy
equality so we have already observed that, in the autonomous case,
the existence of a global attractor might be established by using the
notion of generalized semiflow introduced by J.M.~Ball (see
\cite[Rem.~7]{CFG}). This is exactly the first (and main) result of this
contribution. Namely, if $d=2$ and $h$ does not depend on time, we
prove that \eqref{sy1}-\eqref{sy4} with \eqref{nslip}-\eqref{sy6}
defines a generalized semiflow which is point dissipative and
possesses a compact attractor. An interesting consequence is that we
can also prove the existence of a global attractor for the nonlocal
Cahn-Hilliard equation with convection assuming $u\in L^{\infty}(\Omega)^d$ is given and independent of time.
This can be achieved even in
the case $d=3$ with a restriction on the growth of $F$ (still including
the classical smooth double-well potential). The reason is that, for the
Cahn-Hilliard equation alone, the energy equality also holds in three
dimensions. In addition, in this case, we can prove uniqueness so that
we can define a semiflow and the related global attractor is connected.
The last result of this paper is of interest, in particular, for the three dimensional nonautonomous case.
Indeed, we first demonstrate a suitable generalization of an integral form of Gronwall's lemma.
This inequality allows us to show that any weak solution satisfies a dissipative estimate also
in dimension three.
Moreover, we can show that there is a weak solution satisfying the energy estimate for
any initial time on, with some growth restrictions on $F$ if $d=3$.
Using this fact we can establish the existence of the trajectory attractor
following the theory presented in \cite{CV} (cf. \cite{GG2} for the local Cahn-Hilliard-Navier-Stokes system).

The plan of the paper goes as follows. In the next Section \ref{sec2} we
introduce the assumptions and we briefly restate the results obtained in
\cite{CFG}. Then, in Section \ref{sec3}, we proceed to proving the
main result by recalling first some basic notions on generalized
semiflows. The convective nonlocal Cahn-Hilliard equation case is discussed in
Section \ref{sec4}, while the generalized Gronwall lemma and the dissipative estimate
are proven in Section \ref{sec5}. The final Section \ref{sec6} is devoted
to the existence of the trajectory attractor.

\section{Functional setup and known results}\setcounter{equation}{0}
\label{sec2}

For $d=2,3$ we introduce the classical Hilbert
spaces for the Navier-Stokes equations (see, e.g., \cite{T})
$$G_{div}:=\overline{\{u\in C^\infty_0(\Omega)^d:\mbox{ div}(u)=0\}}^{L^2(\Omega)^d},$$
and
$$V_{div}:=\{u\in H_0^1(\Omega)^d:\mbox{ div}(u)=0\}.$$
We also set $H=L^2(\Omega)$, $V=H^1(\Omega)$ and denote by
$\|\cdot\|$ and $(\cdot,\cdot)$ the norm and the scalar product,
respectively, on both $H$ and $G_{div}$. $H$ will also be used for
$L^2$ spaces of vector or matrix valued functions. The notation
$\langle\cdot,\cdot\rangle$  will stand for the duality pairing
between a Banach space and its dual.  $V_{div}$ is endowed with the
scalar product
$$(u,v)_{V_{div}}=(\nabla u,\nabla v),\qquad\forall u,v\in V_{div}.$$
Let us also recall the definition of the Stokes operator  $A:D(A)\cap
G_{div}\to G_{div}$ in the case of no-slip boundary
condition \eqref{nslip}, i.e. $A=-P\Delta$ with domain $D(A)=H^2(\Omega)^d\cap V_{div}$,
where $P:L^2(\Omega)^d\to G_{div}$ is the Leray projector. Notice that we have
$$(Au,v)=(u,v)_{V_{div}}=(\nabla u,\nabla v),\qquad\forall u\in D(A),\quad\forall v\in V_{div}.$$
We also recall that $A^{-1}:G_{div}\to G_{div}$ is a self-adjoint compact operator in $G_{div}$
and by the classical spectral theorems there exists a sequence $\lambda_j$ with $0<\lambda_1\leq\lambda_2\leq\cdots$ and $\lambda_j\to\infty$,
and a family of $w_j\in D(A)$ which is orthonormal in $G_{div}$ and such that $Aw_j=\lambda_jw_j$.
We also define the map $\mathcal{A}:V_{div}\times H\to V_{div}'$ in the following way.
For every $u\in V_{div}$ and every $\varphi\in H$ we set
$$\langle\mathcal{A}(u,\varphi),v\rangle:=(\nu(\varphi)2Du,Dv),\qquad\forall v\in V_{div},$$
where $\nu$ is a continuous function satisfying $\nu_1\leq\nu(s)\leq\nu_2$,
for all $s\in\mathbb{R}$, with $\nu_1,\nu_2>0$. Notice that if $\nu=1$ we have
$$\langle\mathcal{A}(u,\varphi),v\rangle=(2Du,Dv)=(\nabla u,\nabla v),\qquad\forall u, v\in V_{div},$$
and hence in this case we have $\mathcal{A}(u,\varphi)=Au$ for every $u\in D(A)$. Moreover we have
$$\Vert\mathcal{A}(u,\varphi)\Vert_{V_{div}'}\leq\nu_2\Vert u\Vert_{V_{div}},\qquad\forall u\in V_{div},\quad
\forall\varphi\in H.$$
The trilinear form $b$ which appears in the weak formulation of the
Navier-Stokes equations is defined as usual
$$b(u,v,w)=\int_{\Omega}(u\cdot\nabla)v\cdot w,\qquad\forall u,v,w\in V_{div},$$
and the associated bilinear map $\mathcal{B}$ from $V_{div}\times V_{div}$ into $V_{div}'$
as
$$\langle\mathcal{B}(u,v),w\rangle=b(u,v,w),\qquad\forall u,v,w\in V_{div}.$$
We shall need the following standard estimates which hold for all $u\in V_{div}$
\begin{align}
&\|\mathcal{B}(u,u)\|_{V_{div}'}\leq c\|\nabla u\|^{3/2}\|u\|^{1/2},
\qquad d=3,\label{standest3D}\\
&\|\mathcal{B}(u,u)\|_{V_{div}'}\leq c\|u\|\|\nabla u\|,\qquad d=2.\label{standest2D}
\end{align}

The assumptions listed below are the same as in \cite{CFG}. We report
them for the reader's convenience.
 \begin{description}
 \item[(A1)] $J\in W^{1,1}(\mathbb{R}^d),\quad
     J(x)=J(-x),\quad a\geq 0\quad\mbox{a.e. in } \Omega$.
 \item[(A2)]The function $\nu$ is locally Lipschitz on $\mathbb{R}$ and there exist $\nu_1,\nu_2>0$ such that
 $$\nu_1\leq\nu(s)\leq \nu_2,\qquad\forall s\in\mathbb{R}.$$
 \item[(A3)] $F\in C^{2,1}_{loc}(\mathbb{R})$ and there exists $c_0>0$
     such that
             $$F^{\prime\prime}(s)+a(x)\geq c_0,\qquad\forall s\in\mathbb{R},\quad\mbox{a.e. }x\in\Omega.$$
 \item[(A4)] There exist
     $c_1>\frac{1}{2}\|J\|_{L^1(\mathbb{R}^d)}$ and
     $c_2\in\mathbb{R}$ such that
             $$F(s)\geq c_1s^2-c_2,\qquad\forall s\in\mathbb{R}.$$
 \item[(A5)] There exist $c_3>0$, $c_4\geq0$ and $p\in(1,2]$
     such that
             $$|F^\prime(s)|^p\leq c_3|F(s)|+c_4,\qquad
             \forall s\in\mathbb{R}.$$
\end{description}


\begin{oss}
{\upshape
Since $F$ is bounded from below, it is easy to see that (A5) implies
that $F$ has polynomial growth of order $p'$, where $p'\in[2,\infty)$
is the conjugate index to $p$. Namely, there exist $c_5>0$ and
$c_6\geq 0$ such that
\begin{equation}
\label{growth}
|F(s)|\leq c_5|s|^{p'}+c_6,\qquad\forall s\in\mathbb{R}.
\end{equation}
Observe that assumption (A5) is fulfilled by a potential of arbitrary
polynomial growth. For example, (A3)--(A5) are satisfied for the case
of the well-known double-well potential $F(s)=(s^2-1)^2$.}
\label{Fgrowth}
\end{oss}

We also recall the notion of weak solution to
system
\eqref{sy1}-\eqref{sy4} with \eqref{nslip}-\eqref{sy6}.
\begin{defn}
\label{wfdfn} Let $T>0$, $h\in L^2(0,T;V_{div}')$, $u_0\in
G_{div}$, $\varphi_0\in H$ with $F(\varphi_0)\in L^1(\Omega)$ be
given. Then $[u,\varphi]$ is a weak solution to
\eqref{sy1}-\eqref{sy4} on $(0,T)$ satisfying
\eqref{nslip}-\eqref{sy6} if
\begin{itemize}
\item

$u$, $\varphi$ and $\mu$ satisfy

\begin{align}
&u\in L^{\infty}(0,T;G_{div})\cap L^2(0,T;V_{div}),\label{df1}\\
&u_t\in L^{4/3}(0,T;V_{div}'),\qquad\mbox{if}\quad d=3,\nonumber\\
&u_t\in L^{2-\gamma}(0,T;V_{div}'),\qquad\forall\gamma\in(0,1),
\quad\mbox{if}\quad d=2,\nonumber\\
&\varphi\in L^{\infty}(0,T;H)\cap L^2(0,T;V),\nonumber\\
&\varphi_t\in L^{4/3}(0,T;V'),\quad\mbox{if}\quad d=3,\label{df5}\\
&\varphi_t\in L^{2-\delta}(0,T;V'),
\quad\forall\delta\in(0,1),\quad\mbox{if}\quad d=2,\label{df6}\\
&\mu\in L^2(0,T;V); \nonumber
\end{align}

\item we have
\begin{equation}
\label{chempot}
\mu=a\varphi-J\ast\varphi+F'(\varphi),
\end{equation}
and for every $\psi\in V$, every $v\in V_{div}$ and for almost
any $t\in(0,T)$
\begin{align}
&\langle\varphi_t,\psi\rangle+(\nabla\mu,\nabla\psi)
=(u,\varphi\nabla\psi),\label{wf1}\\
&\langle u_t,v\rangle+(\nu(\varphi)2Du,Dv)+b(u,u,v)
=-(v,\varphi\nabla\mu)+\langle h,v\rangle;
\label{wf2}
\end{align}
\item the following initial conditions hold
\begin{equation}
u(0)=u_0,\qquad\varphi(0)=\varphi_0.\label{ic}
\end{equation}
\end{itemize}
\end{defn}

\begin{oss}
\label{mass} {\upshape As a consequence, the total concentration is conserved.
Indeed, take $\psi=1$ in \eqref{wf1} so
that $\langle\varphi_t,1\rangle=0$ and $(\varphi(t),1)=
(\varphi_0,1)$ for all $t\in[0,T]$.}
\end{oss}

\begin{oss}
{\upshape The initial conditions \eqref{ic} are meant in the weak
sense. Indeed we have $u\in C_w([0,T];G_{div})$ and $\varphi\in
C_w([0,T];H)$. }
\end{oss}

Assumptions (A1)--(A5) are enough to establish the existence of a
global weak solution \cite{CFG}. However, to prove the results of this
paper, we shall need to replace (A4) with the following stronger
assumption (compare with \cite[(A2)]{BH1}).
\begin{description}
\item[(A6)] $F\in C^2(\mathbb{R})$ and there exist $c_7>0$,
    $c_8>0$ and $q>0$ such that
            $$F^{\prime\prime}(s)+a(x)\geq c_7\vert s\vert^{2q} - c_8,
            \qquad\forall s\in\mathbb{R},\quad\mbox{a.e. }x\in\Omega.$$
\end{description}
Thanks to (A6) further regularity properties for $\varphi$,
$\varphi_t$, $u_t$ can be established and, in particular, the energy
identity in two dimensions can be obtained.
For this reason, in the case assumption (A6) holds, it is convenient to introduce the following
\begin{defn}
Suppose (A6) holds and let
 $T>0$, $h\in L^2(0,T;V_{div}')$, $u_0\in
 G_{div}$, $\varphi_0\in H$ with $F(\varphi_0)\in L^1(\Omega)$ be
 given. A couple $[u,\varphi]$ is a weak solution to
 \eqref{sy1}-\eqref{sy4} on $(0,T)$ satisfying
 \eqref{nslip}-\eqref{sy6} if $[u,\varphi]$ is a weak solution
 in the sense of Definition \ref{wfdfn} satisfying the further regularity property
\begin{equation}
\varphi \in L^\infty(0,T;L^{2+2q}(\Omega)).
 \label{impr0}
\end{equation}
\label{wfdfn2}
\end{defn}

 Summing up, the main results of
\cite{CFG} are contained in the following
\begin{thm}
\label{thm} Let $h\in L^2_{loc}([0,\infty);V^\prime_{div})$,
$u_0\in G_{div}$, $\varphi_0\in H$ such that $F(\varphi_0)\in
L^1(\Omega)$ and suppose that (A1)-(A5) are satisfied. Then, for
every given $T>0$, there exists a weak solution $[u,\varphi]$ (in the sense
of Definition \ref{wfdfn}) which satisfies the following energy
inequality for almost all $t>0$
\begin{equation}
\mathcal{E}(u(t),\varphi(t)) +\int_0^t\Big(2\|\sqrt{\nu(\varphi)}Du \|^2+\|\nabla\mu \|^2\Big)d\tau
\leq\mathcal{E}(u_0,\varphi_0)+\int_0^t\langle h(\tau),u \rangle d\tau,\label{ei}
\end{equation}
where we have set
$$\mathcal{E}(u(t),\varphi(t))=\frac{1}{2}\|u(t)\|^2+\frac{1}{4}
\int_{\Omega}\int_{\Omega}J(x-y)(\varphi(x,t)-\varphi(y,t))^2
dxdy+\int_{\Omega}F(\varphi(t)).$$ If (A6) holds in place of (A4)
then we also have
\begin{itemize}
\item there exists a weak solution $[u,\varphi]$ (in the sense of
    Definition \ref{wfdfn2}) such that
  \begin{align}
 &\varphi_t\in L^2(0,T;V'),\quad\mbox{if}\quad d=2\quad\mbox{ or }
 \quad d=3 \mbox{ and } q\geq 1/2,\label{impr2}\\
 & u_t\in L^2(0,T;V_{div}'),\quad\mbox{if}\quad d=2,\label{u_tnew}
 \end{align}
which still satisfies the energy inequality \eqref{ei} for almost all
$t>0$;
\item if $d=2$ then any weak solution (in the sense of Definition
    \ref{wfdfn2}) is such that
\begin{equation}
u\in C([0,\infty);G_{div}),\qquad\varphi\in C([0,\infty);H),
\nonumber
\end{equation}
and
\begin{equation}
\frac{d}{dt}\mathcal{E}(u,\varphi)
+2\|\sqrt{\nu(\varphi)}Du \|^2+\|\nabla\mu\|^2=\langle h(t),u\rangle,
\label{idendiffcor}
\end{equation}
i.e., \eqref{ei} with the equal sign holds for every $t\geq 0$; in
addition, if
$$
\|h\|_{L^2_{tb}(0,\infty;V_{div}')}:=\Big(\sup_{t\geq 0}
\int_t^{t+1}\|h(\tau)\|_{V_{div}'}^2 d\tau\Big)^{1/2}
<\infty
$$
then the following dissipative estimate is satisfied
\begin{equation}
\mathcal{E}(u(t),\varphi(t))\leq \mathcal{E}(u_0,\varphi_0)e^{-kt}+ F(m_0)|\Omega| + K,
\qquad\forall t\geq 0,\label{dissest}
\end{equation}
where $m_0=(\varphi_0,1)$ and $k$, $K$ are two positive constants
which are independent of the initial data, with $K$ depending on
$\Omega$, $\nu_1$, $J$, $F$,
$\|h\|_{L^2_{tb}(0,\infty;V_{div}')}$.
\end{itemize}
\label{exist}
\end{thm}

\begin{oss}
{\upshape
If $u\in C_w([0,T];G_{div})$ and $\varphi\in
C_w([0,T];H)$ are the weakly continuous representatives of the global weak solution $z=[u,\varphi]$
given by Theorem \ref{exist}, then the energy inequality \eqref{ei} holds
also for {\itshape all} $t\geq 0$ (see Lemma \ref{templsc} below).
}
\end{oss}

We conclude by observing that it is straightforward to deduce  from
Theorem \ref{thm} the following result for the convective nonlocal Cahn-Hilliard equation
with a given velocity field.
\begin{cor}
\label{NLCH1} Let $u\in L^2_{loc}([0,\infty);V_{div}\cap
L^\infty(\Omega)^d)$ be given and let $\varphi_0\in H$ be such
that $F(\varphi_0)\in L^1(\Omega)$. Suppose that (A1), (A3),
(A5) and (A6) (with $q \geq \frac{1}{2}$ if $d=3$) are satisfied. Then,
for every $T>0$, there exists a weak solution $\varphi \in
L^2(0,T;V)\cap H^1(0,T;V^\prime)$ to
\eqref{chempot}-\eqref{wf1} on $[0,T]$ such that
$\varphi(0)=\varphi_0$ and $(\varphi(t),1) = (\varphi_0,1)$ for all
$t\in[0,T]$. In addition, the following energy identity holds for all
$t\geq 0$
\begin{equation}
\frac{d}{dt}\left(\frac{1}{4}
\int_{\Omega}\int_{\Omega}J(x-y)(\varphi(x,t)-\varphi(y,t))^2 dxdy
+\int_{\Omega}F(\varphi(t))\right) +\|\nabla\mu\|^2 =
(u\varphi, \nabla\mu).
\label{energyCH}
\end{equation}
\end{cor}


\section{Global attractor in 2D}\setcounter{equation}{0}
\label{sec3}


We first report for the reader's convenience some basic definitions and results from the theory
of generalized semiflows (see \cite{Ba}).

Let $\mathcal{X}$ be a metric space (not necessarily complete)  with
metric $\mathbf{d}$. For any $A,B\subset\mathcal{X}$ the
Hausdorff semidistance between $A$ and $B$ is
dist$(A,B):=\sup_{a\in A}\inf_{b\in B}\mathbf{d}(a,b)$.

\begin{defn}
A {\it generalized semiflow} $\mathcal{G}$ on $\mathcal{X}$ is a family of
maps $z:[0,\infty)\to\mathcal{X}$ satisfying the following hypothesis
\begin{description}
\item[(H1)] ({\it Existence}) For each $z_0\in\mathcal{X}$ there exists at least one $z\in\mathcal{G}$
with $z(0)=z_0$.
\item[(H2)]({\it Translates of solutions are solutions}) If $z\in\mathcal{G}$ and $\tau\geq 0$,
then $z^{\tau}\in\mathcal{G}$, where $z^{\tau}(t):=z(t+\tau)$, for every $t\geq 0$.
\item[(H3)]({\it Concatenation}) If $z_1,z_2\in\mathcal{G}$ and $\tau\geq 0$, with $z_1(\tau)=z_2(0)$,
setting
$$
z(t):= \left\{ \begin{array}{ll}
z_1(t) & \textrm{if}\quad 0\leq t\leq\tau,\\
z_2(t) & \textrm{if}\quad t>\tau,
\end{array} \right.
$$
then $z\in\mathcal{G}$.
\item[(H4)]({\it Upper semicontinuity with respect to initial data})
If $z_j\in\mathcal{G}$ with $z_j(0)\to z_0$, then there exist a subsequence $\{z_{j_k}\}$
of $\{z_j\}$ and $z\in\mathcal{G}$ with $z(0)=z_0$
such that $z_{j_k}(t)\to z(t)$ for each $t\geq 0$.
\end{description}
\label{defsemgen}
\end{defn}

If $\mathcal{G}$ is a generalized semiflow and $E\subset\mathcal{X}$, we define for every $t\geq 0$
$$T(t)E=\{z(t): z\in\mathcal{G}\mbox{ with }z(0)\in E\}.$$
The {\it positive orbit} of $z\in\mathcal{G}$ is the set $\gamma^{+}(z)=\{z(t): t\geq 0\}$.
If $E\subset\mathcal{X}$, then the {\it positive orbit} of $E$ is the set
$\gamma^{+}(E)=\cup_{t\geq 0}T(t)E$. For $\tau\geq 0$ we also set
$$\gamma^{\tau}(E)=\bigcup_{t\geq\tau}T(t)E=\gamma^{+}(T(\tau)E).$$
The $\omega-${\it limit} of $z\in\mathcal{G}$ is the set
$$\omega(z):=\{w\in\mathcal{X}: z(t_j)\to w\mbox{ for some sequence }t_j\to\infty\}.$$
If $E\subset\mathcal{X}$ the $\omega-${\it limit} of $E$ is the set
$$\omega(E):=\{w\in\mathcal{X}:\exists z_j\in\mathcal{G},
z_j(0)\in E,z_j(0)\mbox{ bounded, and }\exists t_j\to\infty\mbox{ s.t. }
z_j(t_j)\to w\}.$$

The subset $\mathcal{A}$ is a {\it global attractor} for the
generalized semiflow $\mathcal{G}$ if $\mathcal{A}$ is compact,
invariant, i.e. $T(t)\mathcal{A}=\mathcal{A}$ for all $t\geq 0$, and attracts all
bounded subsets of $\mathcal{X}$, i.e. dist$(T(t)B,\mathcal{A})\to 0$
as $t\to\infty$, for every bounded set $B\subset\mathcal{X}$.

The generalized semiflow $\mathcal{G}$ is {\it eventually bounded} if, given
any bounded set $B\subset\mathcal{G}$, there exists $\tau\geq 0$ such that $\gamma^{\tau}(B)$
is bounded.

$\mathcal{G}$ is {\it point dissipative} if there is a bounded set $\mathcal{B}_0$ such that
for any $z\in\mathcal{G}$ there exists $t_0=t_0(z)\geq 0$ such that $z(t)\in\mathcal{B}_0$
for all $t\geq t_0$.

$\mathcal{G}$ is {\it asymptotically compact} if for any sequence $z_j\in\mathcal{G}$ with
$z_j(0)$ bounded, and any sequence $t_j\to\infty$, the sequence $z_j(t_j)$ is precompact.

$\mathcal{G}$ is {\it compact} if for any sequence $z_j\in\mathcal{G}$ with $z_j(0)$ bounded
there exists a subsequence $z_{j_k}$ such that $z_{j_k}(t)$ converges for every $t>0$.

\begin{prop}
Let $\mathcal{G}$ be asymptotically compact. Then  $\mathcal{G}$
is eventually bounded.
\end{prop}

\begin{prop}
Let $\mathcal{G}$ be eventually bounded and compact. Then $\mathcal{G}$ is asymptotically compact.
\label{asymptcompact}
\end{prop}

\begin{thm}
A generalized semiflow $\mathcal{G}$ has a global attractor if and only if
$\mathcal{G}$ is point dissipative and asymptotically compact.
The global attractor $\mathcal{A}$ is unique and given by
$$\mathcal{A}=\bigcup\{\omega(B): B\mbox{ is a bounded subset of }\mathcal{X}\}=\omega(\mathcal{X}).$$
Furthermore $\mathcal{A}$ is the maximal compact invariant subset of $\mathcal{X}$.
\label{attractor}
\end{thm}



We now turn to our system \eqref{sy1}-\eqref{sy4} endowed with
\eqref{nslip} in the case $d=2$. Also, we suppose that $h$ is time
independent, i.e.,
\begin{equation}
\label{force}
h\in V^\prime_{div}.
\end{equation}
We first have to choose a suitable metric space where the weak
solutions can be defined in order to construct the associated
generalized semiflow.

We therefore fix $m\geq0$ and introduce the metric space
$$\mathcal{X}_m:=G_{div}\times\mathcal{Y}_m,$$
where
\begin{equation}
\label{phsp}
\mathcal{Y}_m:=\{\varphi\in H\,:\, F(\varphi)\in L^1(\Omega),
\; |(\varphi,1)|\leq m\},
\end{equation}
endowed with the metric
\begin{equation*}
\mathbf{d}(z_1,z_2)=\|u_1-u_2\|+\|\varphi_1-\varphi_2\|
+\Big|\int_{\Omega}F(\varphi_1)-\int_{\Omega}F(\varphi_2)\Big|^{1/2},
\end{equation*}
for every $z_1:=[u_1,\varphi_1]$ and $z_2:=[u_2,\varphi_2]$ in $\mathcal{X}_m$.

On account of Theorem \ref{exist}, let us now denote by
$\mathcal{G}$ the set of all weak solutions in the sense of
Definition \ref{wfdfn2} (we shall assume (A6)) corresponding to all initial
data $z_0:=[u_0,\varphi_0]\in\mathcal{X}_m$. Our aim is to prove
that $\mathcal{G}$ is a generalized semiflow on $\mathcal{X}_m$.

\begin{prop}
Let $d=2$. Suppose that (A1)-(A3), (A5), (A6) and \eqref{force} hold.
Then $\mathcal{G}$ is a generalized semiflow on $\mathcal{X}_m$.
\label{gensem}
\end{prop}

\begin{proof}
It is immediate to see that $\mathcal{G}$ satisfies (H1)-(H3) of
Definition \ref{defsemgen}. The only property which is not trivial to
prove is (H4). We therefore consider a sequence $\{z_j\}$, with
$z_j:=[u_j,\varphi_j]$, of weak solutions (cf. Definition \ref{wfdfn2})
such that $z_j(0):=[u_{j0},\varphi_{j0}]\to z_0:=[u_0,\varphi_0]$
in $\mathcal{X}_m$. Since every weak solution satisfies the energy
identity, for each $j\in\mathbb{N}$ and for every $t\geq 0$ we can
write
\begin{equation}
\mathcal{E}(z_j(t))+\int_0^t\Big(2\|\sqrt{\nu(\varphi_j)}Du_j\|^2+\|\nabla\mu_j \|^2\Big)d\tau
=\mathcal{E}(z_{j0})
+\int_0^t\langle h,u_j \rangle d\tau,\label{eneq}
\end{equation}
where $z_{j0}:=z_j(0)$. From this identity, by recalling the definition
of the energy functional $\mathcal{E}$ and using (A1)-(A3), (A5)-(A6) we deduce that $\{u_j\}$ is
bounded in $L^{\infty}(0,T; G_{div})\cap L^2(0,T;V_{div})$ for
every $T>0$, $\{\varphi_j\}$ is bounded in $L^{\infty}(0,T;H)\cap
L^2(0,T;V)$ for every $T>0$ and $\{\mu_j\}$ is bounded in
$L^2(0,T;V)$ for every $T>0$ (cf. \cite{CFG} for details). From equations \eqref{wf1} and
\eqref{wf2}, written for each weak solution $[u_j,\varphi_j]$, and
arguing as in \cite{CFG} we also show that $\{u_j'\}$ is bounded in
$L^2(0,T;V_{div}')$ for every $T>0$ and that $\{\varphi_j'\}$ is
bounded in $L^2(0,T;V')$ for every $T>0$. Therefore, we deduce that
there exist $u\in L^{\infty}(0,T; G_{div})\cap L^2(0,T;V_{div})$ for
every $T>0$, $\varphi\in L^{\infty}(0,T;H)\cap L^2(0,T;V)$ for
every $T>0$ and $\mu\in L^2(0,T;V)$ for every $T>0$ such that, for
a subsequence that  we do not relabel, we have
\begin{align}
&u_j\rightharpoonup u\quad
\mbox{weakly}^{\ast}\mbox{ in }L^{\infty}(0,T;G_{div})
\;\mbox{and weakly in }L^2(0,T;V_{div}),\label{conv1}\\
&u_j'\rightharpoonup u'\quad\mbox{weakly in }L^2(0,T;V_{div}'),
\label{conv2}\\
&u_j\to u\quad\mbox{strongly in }L^2(0,T;G_{div}),\label{conv3}\\
&\varphi_j\rightharpoonup\varphi\quad
\mbox{weakly}^{\ast}\mbox{ in }L^{\infty}(0,T;H)
\;\mbox{and weakly in }L^2(0,T;V),\label{conv4}\\
&\varphi_j'\rightharpoonup\varphi'\quad\mbox{weakly in }L^2(0,T;V'),\label{conv5}\\
&\varphi_j\to\varphi\quad\mbox{strongly in }L^2(0,T;H),\label{conv6}\\
&\mu_j\rightharpoonup\mu\quad\mbox{weakly in }L^2(0,T;V).
\label{conv7bis}
\end{align}
From \eqref{conv2} and \eqref{conv5} we obtain
\begin{align}
&u_j(t)\rightharpoonup u(t)\quad\mbox{weakly in }G_{div},
\quad\forall t\geq 0,\label{conv7}\\
&\varphi_j(t)\rightharpoonup\varphi(t)\quad\mbox{weakly in }H,
\quad\forall t\geq 0.\label{conv8}
\end{align}
Indeed, for every $v\in V_{div}$ and every $t\geq 0$, we have
\begin{equation*}
\int_0^T\langle u_j'(\tau)-u'(\tau),v\rangle\chi_{[0,t]}(\tau)d\tau
=(u_j(t)-u(t),v)-(u_{j0}-u_0,v)\to 0\nonumber
\end{equation*}
as $j\to\infty$. Hence $(u_j(t),v)\to(u(t),v)$, for every $v\in
V_{div}$ and every $t\geq 0$ so that \eqref{conv7} follows from the
density of $V_{div}$ in $G_{div}$ and from the boundedness of the
sequence of $u_j$ in $L^{\infty}(0,T;H)$ for every $T>0$. By the
same argument we get \eqref{conv8}. By means of the convergences
above and of the fact that each $z_j$ is a weak solution, by passing to
the limit in the variational formulation for $z_j=[u_j,\varphi_j]$ we
infer that $z=[u,\varphi]$ is a weak solution as well. Furthermore,
from \eqref{conv7} and \eqref{conv8} we get $z(0)=z_0$. We are
now left to prove the convergence in $\mathcal{X}_m$ for each time
$t\geq 0$. In order to do that, let us represent the potential $F$ in the
following form
\begin{equation}
F(s)=G(x,s)-\Big(a(x)-\frac{c_0}{2}\Big)\frac{s^2}{2},\label{pot}
\end{equation}
where, due to (A3), function $G(x,\cdot)$ is strictly convex in
$\mathbb{R}$ for almost every $x\in\Omega$. By means of
\eqref{pot} the energy $\mathcal{E}$ can be rewritten in the form
\begin{equation}
\mathcal{E}(z)=\frac{1}{2}\|u\|^2+\frac{c_0}{4}\|\varphi\|^2-\frac{1}{2}(\varphi,J\ast\varphi)
+\int_{\Omega}G(x,\varphi(x))dx,
\end{equation}
for every $z=[u,\varphi]\in\mathcal{X}_m$.
As a consequence of the weak convergences \eqref{conv7} and \eqref{conv8} we see that we have
\begin{equation}
\liminf_{j\to\infty}\mathcal{E}(z_j(t))\geq\mathcal{E}(z(t)),\quad\forall t\geq 0.
\label{liminfE}
\end{equation}
Indeed \eqref{liminfE} follows from the weak lower semicontinuity in
$H$ of the $L^2-$norm and of the convex integral functional in $G$,
and from the compactness of the convolution operator
$J\ast\cdot:H\to H$ (cf. (A1)). Recall that if we consider the
functional $\mathcal{L}:H\to\mathbb{R}\cup\{+\infty\}$, where
$$\mathcal{L}(\varphi):=\int_{\Omega}G(x,\varphi(x))dx,$$ for
every $\varphi\in H$ such that $G(\cdot,\varphi(\cdot))\in
L^1(\Omega)$ ($\mathcal{L}(\varphi)=+\infty$ otherwise), due to
the convexity of $G(x,\cdot)$ for a.e. $x\in\Omega$ and to the lower
bound $G(x,s)\geq-\alpha s^2-\beta$, for every $s\in\mathbb{R}$
and for some $\alpha,\beta\geq 0$, then $\mathcal{L}$ is weakly
lower-semicontinuous in $H$.
\\
Since each weak solution satisfy the energy equation
\eqref{eneq}, we have
\begin{align}
&\limsup_{j\to\infty}\mathcal{E}(z_j(t))=\mathcal{E}(z_0)-\liminf_{j\to\infty}\int_0^t
\Big(2\|\sqrt{\nu(\varphi_j)}Du_j \|^2+\|\nabla\mu_j \|^2\Big)d\tau+\int_0^t\langle h,u \rangle d\tau
\nonumber\\
&\leq \mathcal{E}(z(0))-\int_0^t\Big(2\|\sqrt{\nu(\varphi)}Du \|^2+\|\nabla\mu \|^2\Big)d\tau
+\int_0^t\langle h,u \rangle d\tau\nonumber\\
&=\mathcal{E}(z(t)),\qquad\forall t\geq 0,
\label{limsupE}
\end{align}
due to \eqref{conv1}, \eqref{conv7bis} and on account of the fact
that, since $z_j(0)\to z_0=z(0)$ in $\mathcal{X}_m$, then
$u_j(0)\to u_0=u(0)$ in $G_{div}$,
$\varphi_j(0)\to\varphi_0=\varphi(0)$ in $H$ and
$\int_{\Omega}F(\varphi_j(0))\to
\int_{\Omega}F(\varphi_0)=\int_{\Omega}F(\varphi(0))$.
We have also used the fact that
\begin{equation}
\sqrt{\nu(\varphi_j)}Du_j\rightharpoonup\sqrt{\nu(\varphi)}Du\qquad\mbox{weakly in }L^2(H).
\label{weakvisc}
\end{equation}
This convergence easily follows from
the uniform bound $\|\sqrt{\nu(\varphi_j)}\|_{\infty}\leq\sqrt{\nu_2}$,
the strong convergence $\sqrt{\nu(\varphi_j)}\to\sqrt{\nu(\varphi)}$ in $L^2(H)$
and the weak convergence \eqref{conv1}  (see \cite{CFG} for details).
Therefore
$\mathcal{E}(z_j(0))\to\mathcal{E}(z(0))$ as $j\to\infty$. Hence,
from \eqref{liminfE} and \eqref{limsupE} we get
\begin{equation}
\mathcal{E}(z_j(t))\to\mathcal{E}(z(t))\quad\mbox{as }j\to\infty,\quad\forall t\geq 0.
\label{limE}
\end{equation}
We know that $(\varphi_j(t), J\ast\varphi_j(t))\to(\varphi(t),
J\ast\varphi(t))$ for every $t\geq 0$. Then \eqref{limE} yield
\begin{align}
&\frac{1}{2}\|u_j(t)\|^2+\frac{c_0}{4}\|\varphi_j(t)\|^2+\int_{\Omega}G(x,\varphi_j(x,t))dx
\to\frac{1}{2}\|u(t)\|^2\nonumber\\
&+\frac{c_0}{4}\|\varphi(t)\|^2+\int_{\Omega}G(x,\varphi(x,t))dx,\quad\mbox{as }j\to\infty,\quad\forall t\geq 0.
\nonumber
\end{align}
Therefore, for all $t\geq 0$, we have (cf. also \eqref{conv7}-
\eqref{conv8})
\begin{align}
&\frac{1}{2}\|u_j(t)-u(t)\|^2+\frac{c_0}{4}\|\varphi_j(t)-\varphi(t)\|^2
+\int_{\Omega}G(x,\varphi_j(x,t))dx-\int_{\Omega}G(x,\varphi(x,t))dx\nonumber\\
&=\frac{1}{2}\|u_j(t)\|^2+\frac{c_0}{4}\|\varphi_j(t)\|^2+\int_{\Omega}G(x,\varphi_j(x,t))dx
-(u_j(t),u(t))\nonumber\\
&-\frac{c_0}{2}(\varphi_j(t),\varphi(t))+\frac{1}{2}\|u(t)\|^2+\frac{c_0}{4}\|\varphi(t)\|^2
-\int_{\Omega}G(x,\varphi(x,t))dx\to 0,
\label{conv}
\end{align}
as $j\to\infty$. Therefore we infer that
\begin{equation}
\limsup_{j\to\infty}\int_{\Omega}G(x,\varphi_j(x,t))dx
\leq\int_{\Omega}G(x,\varphi(x,t))dx,\quad\forall t\geq 0,\nonumber
\end{equation}
and, due to the $H-$weak lower semicontinuity of the integral
functional $\mathcal{L}$, we obtain
\begin{equation}
\int_{\Omega}G(x,\varphi_j(x,t))dx\to\int_{\Omega}G(x,\varphi(x,t))dx,\quad\forall t\geq 0.
\label{convintG}
\end{equation}
From \eqref{conv} and \eqref{convintG} we finally get
\begin{align}
&u_j(t)\to u(t)\quad\mbox{strongly in }G_{div},\quad\forall t\geq 0,\nonumber\\
&\varphi_j(t)\to\varphi(t)\quad\mbox{strongly in }H,\quad\forall t\geq 0,\nonumber
\end{align}
and, on account of \eqref{pot} and \eqref{convintG},  we also have
$$\int_{\Omega}F(\varphi_j(t))\to\int_{\Omega}F(\varphi(t)),\qquad\forall t\geq 0.$$
Hence $z_j(t)\to z(t)$ in $\mathcal{X}_m$, for every $t\geq 0$. We
thus conclude that (H4) holds.
\end{proof}

As a consequence of \eqref{dissest} we have the following

\begin{prop}
Let the hypotheses of Proposition \ref{gensem} hold. Then
$\mathcal{G}$ is point dissipative and eventually bounded.
\label{pointdiss}
\end{prop}

\begin{proof}
Due to (A6) there exists
$\gamma=\gamma(c_7,c_8,J,|\Omega|)\geq 0$ such that
$\mathcal{E}(z)\geq -\gamma$ for every $z\in\mathcal{X}_m$.
Therefore, setting
$\overline{\mathcal{E}}(z):=\mathcal{E}(z)+\gamma\geq 0$,
from \eqref{dissest} we deduce
\begin{align}
&\frac{1}{2}\|u(t)\|^2+\frac{1}{2}\|\sqrt{a}\varphi(t)\|^2-\frac{1}{2}(J\ast\varphi(t),\varphi(t))
+\int_{\Omega}F(\varphi(t))\nonumber\\
&\leq\overline{\mathcal{E}}(z_0)e^{-kt}+L_m,\qquad\forall t\geq 0,
\label{gh}
\end{align}
where $z_0:=[u_0,\varphi_0]$ and $L_m=F(m)|\Omega|+K$.
Now, by using (A6) again we have
\begin{equation}
\frac{1}{2}\|\sqrt{a}\varphi\|^2-\frac{1}{2}(J\ast\varphi,\varphi)
+\int_{\Omega}F(\varphi)\geq c_9\|\varphi\|^2-\gamma,\nonumber
\end{equation}
and therefore
\begin{equation}
\frac{1}{2}\|u(t)\|^2+c_9\|\varphi(t)\|^2\leq\overline{\mathcal{E}}(z_0)e^{-kt}+L_m+\gamma.\label{gh1}
\end{equation}
From \eqref{gh} we infer
\begin{align}
&
\frac{1}{2}\|u(t)\|^2+\int_{\Omega}F(\varphi(t))\leq\overline{\mathcal{E}}(z_0)e^{-kt}+L_m
+c_{10}\|\varphi(t)\|^2\leq c_{11}\overline{\mathcal{E}}(z_0)e^{-kt}+c_{11}L_m+c_{12},
\label{gh2}
\end{align}
where $c_{10}=\frac{1}{2}\|J\|_{L^1}$, $c_{11}=1+c_{10}/c_9$ and $c_{12}=\gamma c_{10}/c_9$.
Therefore, \eqref{gh1} and \eqref{gh2} entail
\begin{align}
&
\|u(t)\|^2+\|\varphi(t)\|^2+\Big|\int_{\Omega}F(\varphi(t))-\int_{\Omega}F(0)\Big|\leq
c_{13}\overline{\mathcal{E}}(z_0)e^{-kt}+c_{13}|L_m|+c_{14},
\label{ok}
\end{align}
for all $t\geq 0$. Here the expressions of the positive constants
$c_{13}$ and $c_{14}$ in terms of the previous constants are
omitted for the sake of simplicity. Setting $z(t):=[u(t),\varphi(t)]$,
\eqref{ok} can be rewritten as follows
\begin{equation}
\mathbf{d}^2(z(t),0)\leq c_{13}\overline{\mathcal{E}}(z_0)e^{-kt}+c_{13}|L_m|+c_{14},\qquad\forall t\geq 0.
\label{ik}
\end{equation}
Choosing therefore $R_0$ such that $R_0^2>c_{13}|L_m|+c_{14}$, from \eqref{ik}
we deduce that
$$\mathbf{d}(z(t),0)\leq R_0,$$
for every $t\geq t_0(z_0)$, where
$$t_0=\frac{1}{k}\log\frac{c_{13}\overline{\mathcal{E}}(z_0)}{R_0^2-(c_{13}|L_m|+c_{14})},$$
which means that $\mathcal{G}$ is point dissipative. By using a similar argument,
\eqref{ok} implies that $\mathcal{G}$ is also eventually bounded.
\end{proof}

We can now prove our main result.

\begin{thm}
Let the hypotheses of Proposition \ref{gensem} hold. Then
$\mathcal{G}$ possesses a global attractor.
\end{thm}

\begin{proof}

By Proposition \ref{pointdiss} we know that the generalized semiflow
$\mathcal{G}$ is point dissipative. Since, again by Proposition
\ref{pointdiss}, $\mathcal{G}$ is also eventually bounded, according
with Theorem \ref{attractor}, we only need to show that
$\mathcal{G}$ is compact (see also Proposition
\ref{asymptcompact}). Let us first observe that the compact
embedding $V\hookrightarrow\hookrightarrow L^{p'}(\Omega)$
and the Aubin-Lions lemma imply
\begin{equation}
\label{Aubin-Lions}
L^2(0,T;V)\cap H^1(0,T,V')\hookrightarrow
\hookrightarrow L^2(0,T;L^{p'}(\Omega)).
\end{equation}
Therefore, from \eqref{conv4} and \eqref{conv5} we deduce that, for
a subsequence that we do not relabel, we have
$$\varphi_j\to\varphi,\qquad\mbox{strongly in }L^2(0,T;L^{p'}(\Omega)),$$
and hence, for a further subsequence, $\varphi_j(t)\to\varphi(t)$
strongly in $L^{p'}(\Omega)$ for a.e. $t\in(0,T)$. Since $F$ has
polynomial growth of order $p'$ (cf. Remark \ref{Fgrowth}), then by
Lebesgue's theorem we deduce
\begin{equation}
\int_{\Omega}F(\varphi_j(t))\to\int_{\Omega}F(\varphi(t)),\qquad\mbox{a.e. }t>0.
\label{Fconv}
\end{equation}
Hence, the strong convergences \eqref{conv3}, \eqref{conv6}, which
imply that for a subsequence we have
\begin{align}
&u_j(t)\to u(t)\quad\mbox{strongly in }G_{div},\quad\mbox{a.e. }t>0,\\
&\varphi_j(t)\to\varphi(t)\quad\mbox{strongly in }H,\quad\mbox{a.e. }t>0,
\end{align}
and \eqref{Fconv} allow to deduce that $\mathcal{E}(z_j(t))\to\mathcal{E}(z(t))$
for almost all $t>0$.
Now, setting
$$\widetilde{\mathcal{E}}(z(t)):=\mathcal{E}(z(t))-\int_0^t\langle h,u(\tau)\rangle d\tau,$$
we still have
$\widetilde{\mathcal{E}}(z_j(t))\to\widetilde{\mathcal{E}}(z(t))$
for almost all $t>0$. Since for each $j$ the function
$\widetilde{\mathcal{E}}(z_j(\cdot))$ is decreasing on $[0,\infty)$
and $\widetilde{\mathcal{E}}(z(\cdot))$ is continuous on
$[0,\infty)$, then
$\widetilde{\mathcal{E}}(z_j(t))\to\widetilde{\mathcal{E}}(z(t))$
for {\it all} $t>0$. Hence
\begin{equation}
\mathcal{E}(z_j(t))\to\mathcal{E}(z(t)),\qquad\forall t>0.
\label{convmathcalE}
\end{equation}
Now, by means of the same argument used to deduce (H4), from
\eqref{convmathcalE} we infer that $z_j(t)\to z(t)$ in
$\mathcal{X}_m$, for all $t>0$. Thus $\mathcal{G}$ is compact.
\end{proof}

\begin{oss}
{\upshape
In the nonautonomous case (say, $h$ depending on time) it would be interesting to establish the
existence of a pullback attractor along the lines of \cite{MRR}
(see also its references), where uniqueness also fails but energy identity holds.
}
\end{oss}


\section{The convective nonlocal Cahn-Hilliard equation}\setcounter{equation}{0}
\label{sec4}

Here we show that the existence of the global attractor for
\eqref{sy1}-\eqref{sy2}, assuming that $u\in L^{\infty}(\Omega)^d$ is given
and independent of time for $d=2,3$, can be
proven arguing as in the previous section.

First, recalling Corollary \ref{NLCH1}, we prove a uniqueness result.

\begin{prop}
\label{NLCH2} Let $u\in L^2(0,T;L^\infty(\Omega)^d\cap
V_{div})$ be given and let $\varphi_0\in H$ be such that
$F(\varphi_0)\in L^1(\Omega)$. Suppose also that (A1), (A3), (A5)
and (A6) (with $q \geq \frac{1}{2}$ if $d=3$) are satisfied. Then,
there exists a unique weak solution $\varphi \in L^2(0,T;V)\cap
H^1(0,T;V^\prime)$ to \eqref{chempot}-\eqref{wf1} on $(0,T)$
such that $\varphi(0)=\varphi_0$.
\end{prop}

\begin{proof}
Suppose that $\varphi_i$, $i=1,2$, are two weak solutions and set
$\varphi=\varphi_1 - \varphi_2$. Then we have
\begin{equation}
\label{NLCH3}
\langle\varphi_t,\psi\rangle
+(\nabla\mu,\nabla\psi)= (u,\varphi\nabla\psi), \qquad\forall\,\psi\in V,
\end{equation}
where (cf. \ref{chempot})
\begin{equation*}
\mu=a\varphi-J\ast\varphi+F'(\varphi_1) - F'(\varphi_2).
\end{equation*}
Note that $(\varphi,1)=0$. Then consider the operator $B_N:=-\Delta$
with domain
$$
D(B_N)=\left\{\phi\in H^2(\Omega)\,:\, \frac{\partial\phi}{\partial n}=0
\; \text{on}\; \partial\Omega\right\}.
$$
and take $\psi=B^{-1}_N\varphi(t)\in D(B_N)$ as test function in
\eqref{NLCH3}. Thus we obtain
\begin{equation*}
\frac{d}{dt} \Vert B_N^{-1/2}\varphi\Vert^2
+ 2(a\varphi-J\ast\varphi+F'(\varphi_1) - F'(\varphi_2),\varphi)=
(u,\varphi\nabla B_N^{-1}\varphi)
\end{equation*}
and, thanks to (A3), we get
\begin{equation}
\label{NLCH5}
\frac{d}{dt} \Vert B_N^{-1/2}\varphi\Vert^2
+ 2c_0\Vert \varphi\Vert^2 \leq 2 \vert ( J\ast\varphi,\varphi)\vert
+ C_1\Vert u\Vert_{L^\infty(\Omega)^d}\Vert \varphi\Vert
\Vert B_N^{-1/2}\varphi\Vert.
\end{equation}
On the other hand, recalling (A1) and using Young's inequality,we
have
\begin{equation}
\label{NLCH6}
\vert ( J\ast\varphi,\varphi)\vert \leq
\Vert B_N^{1/2}(J\ast\varphi)\Vert \Vert B_N^{-1/2}\varphi\Vert
\leq \frac{c_0}{4}\Vert \varphi\Vert^2 +
C_2\Vert B_N^{-1/2}\varphi\Vert^2,
\end{equation}
where $C_2>0$ depends on $c_0$ and on $J$. Then, combining
\eqref{NLCH5} with \eqref{NLCH6} and using once more the Young
inequality, the standard Gronwall lemma entails that $\varphi\equiv
0$.
\end{proof}

A consequence of Corollary \ref{NLCH1} and Proposition
\ref{NLCH2} is that we can define a semiflow $S(t)$ on
$\mathcal{Y}_m$ (cf. \eqref{phsp}) endowed the metric
\begin{equation*}
\bar{\mathbf{d}}(\varphi_1,\varphi_2)=
\|\varphi_1-\varphi_2\|+\Big|\int_{\Omega}
F(\varphi_1)-\int_{\Omega}F(\varphi_2)\Big|^{1/2},
\quad \forall\,\varphi_1,\varphi_2 \in \mathcal{Y}_m,
\end{equation*}
where $m\geq 0$ is given.

We can now prove

\begin{thm}
\label{NLCHattr}
Suppose that (A1), (A3), (A5)
and (A6) (with $q \geq \frac{1}{2}$ if $d=3$) are satisfied
and assume $u\in L^{\infty}(\Omega)^d$ is given and independent of time.
In addition, if $d=3$, suppose that
$p^\prime \in [2,6)$ in \eqref{growth}. Then the dynamical system
$(\mathcal{Y}_m,S(t))$ possesses a connected global attractor.
\end{thm}

\begin{proof}
Observe that the energy identity \eqref{energyCH} entails
$$\mathcal{E}(\varphi(t))+\int_0^t\|\nabla\mu\|^2d\tau=\mathcal{E}(\varphi_0)+\int_0^t(u\varphi,\nabla\mu)d\tau,$$
from which we have
$$\mathcal{E}(\varphi(t))+\frac{1}{2}\int_0^t\|\nabla\mu\|^2d\tau\leq\mathcal{E}(\varphi_0)
+\frac{1}{2}u_{\ast}^2\int_0^t\|\varphi\|^2d\tau,$$
where the energy functional $\mathcal{E}$ is now given by
$$\mathcal{E}(\varphi):=\frac{c_0}{2}\|\varphi\|^2-\frac{1}{2}(\varphi,J\ast\varphi)
+\int_{\Omega}G(x,\varphi(x))dx,$$
and $u_{\ast}:=\|u\|_{L^{\infty}(\Omega)^d}$.
Therefore, the argument used in the previous
section can be adapted to this case. Indeed, in order to prove the compactness of the semiflow $S(t)$
we note that, if $d=3$ and $p^\prime\in(2,6]$, the compact injection \eqref{Aubin-Lions} is valid
and still implies \eqref{Fconv}. Hence, by using the strong convergence $\varphi_j(t)\to\varphi(t)$
in $H$ for a.e $t>0$ we have $\mathcal{E}(\varphi_j(t))\to\mathcal{E}(\varphi(t))$ for a.e. $t>0$.
Setting now
$$
\widetilde{\mathcal{E}}(\varphi(t)):=\mathcal{E}(\varphi(t))-\int_0^t(u\varphi,\nabla\mu)d\tau,
$$
the strong convergence to $\varphi$ in $L^2(H)$ for the sequence $\{\varphi_j\}$ and the weak convergence
to $\mu$ in $L^2(V)$ for the sequence $\{\mu_j\}$ imply that $\widetilde{\mathcal{E}}(\varphi_j(t))\to
\widetilde{\mathcal{E}}(\varphi(t))$ for a.e $t>0$. Since, due to \eqref{energyCH} the function
$\widetilde{\mathcal{E}}(\varphi_j(\cdot))$ is decreasing on $[0,\infty)$ and $\widetilde{\mathcal{E}}(\varphi(\cdot))$ in continuous on $[0,\infty)$, then
$\widetilde{\mathcal{E}}(\varphi_j(t))\to\widetilde{\mathcal{E}}(\varphi(t))$ for {\itshape all}
$t>0$. Hence $\mathcal{E}(\varphi_j(t))\to\mathcal{E}(\varphi(t))$ for all $t>0$
and arguing as in the previous section we get $\varphi_j(t)\to\varphi(t)$ in $\mathcal{Y}_m$ for all $t>0$.
Therefore, the semiflow $S(t)$ is compact. In addition,
the uniqueness of solution trivially implies the Kneser
property (see, e.g., \cite{Ba2,KlVa}) so that the global attractor is also
connected.
\end{proof}

\begin{oss}
{\upshape
The connectedness of the global attractor for the full system remains an open issue.
}
\end{oss}

\section{A dissipative estimate in 3D}\setcounter{equation}{0}
\label{sec5}

In dimension three, of course we are not able to prove an energy identity like
\eqref{idendiffcor}. Actually, one could argue in the spirit of \cite{Ba}, under the
(unproven) assumption that the weak solution $z=[u,\varphi]$ is
strongly continuous from $[0,\infty)$ to $\mathcal{X}_m$. However,
we shall not consider this possibility, but we shall construct
a (generalized) notion of attractor (see next section).
Nevertheless, we are able to prove that a
dissipative estimate like \eqref{dissest} can still be recovered in the
three dimensional case. This the main aim of the present section. We observe that, since
such dissipative estimate relies on the validity of the energy inequality
\eqref{ei} only, then it holds for any weak solution in the sense of Definition \ref{wfdfn2}.

We need the following basic lemma, which is obtained by suitably modifying
\cite[Lemma 7.2]{Ba}.

\begin{lem}
\label{Ballgen}
Let $\theta\in L^1(0,T)$ for every $T>0$ and suppose that
\begin{equation}
\theta(t)+k\int_0^t\theta(\tau)\tau\leq \int_s^t f(\tau)d\tau+\theta(s)+k\int_0^s\theta(\tau)d\tau
\label{intinq}
\end{equation}
holds for a.e. $t$, $s\in(0,\infty)$, with $t\geq s$, where $f\in L^1(0,T)$ for every $T>0$ and
the constant $k\geq 0$ are given. Then we have
\begin{equation}
\theta(t)\leq \theta(s) e^{-k(t-s)}+\int_s^t e^{-k(t-\tau)}f(\tau)d\tau,
\label{thesis1}
\end{equation}
for a.e. $t$, $s\in(0,\infty)$, with $t\geq s$.
Furthermore, suppose $\theta:[0,\infty)\to\mathbb{R}$ is a l.s.c. representative
satisfing \eqref{intinq} for a.e. $t,s\in(0,\infty)$, with $t\geq s$.
Then \eqref{intinq} and \eqref{thesis1} also hold for every $t\geq s$ and for a.e. $s>0$,
and if, in addition, \eqref{intinq} holds for $s=0$,
then we have
\begin{equation}
\theta(t)\leq \theta(0) e^{-kt}+\int_0^t e^{-k(t-\tau)}f(\tau)d\tau,
\label{thesis2}
\end{equation}
for all $t\in[0,\infty)$. In particular, suppose $f(t)=l+g(t)$, where
$l\in\mathbb{R}$ is a given constant and $g\in
L^1_{tb}(0,\infty)$, i.e., $g$ belongs to $L^1_{loc}([0,\infty))$
and is translation bounded, that is,
$$\|g\|_{L^1_{tb}}:=\sup_{t\geq 0}\int_t^{t+1}|g(\tau)|d\tau<\infty.$$
Then we have
\begin{equation}
\theta(t)\leq\theta(0)e^{-kt}+\frac{l}{k}+\frac{\|g\|_{L^1_b}}{1-e^{-k}},
\label{trbd}
\end{equation}
for all $t\in[0,\infty)$.
\label{Ball}
\end{lem}

\begin{proof}
Setting
$$\rho(t)=\theta(t)+k\int_0^t\theta(\tau)d\tau-\int_0^tf(\tau)d\tau,$$
from \eqref{intinq} we have
$\rho(t)\leq\rho(s)$ for a.e. $t$, $s\in(0,\infty)$, with $t\geq s$.
We therefore deduce that
\begin{equation}
\dot{\rho}\leq 0\quad\mbox{in }\mathcal{D}'(0,\infty)
\label{dotrho}
\end{equation}
Indeed, take $\varphi\in\mathcal{D}(0,\infty)$, $\varphi\geq 0$. We have
$$0\leq\int_0^{\infty}\frac{\rho(t)-\rho(t+h)}{h}\varphi(t)dt=\int_0^{\infty}
\rho(t)\frac{\varphi(t)-\varphi(t-h)}{h}dt,$$
for all $h>0$. Letting $h\to 0$, from the previous relation
and by means of Lebesgue's theorem we get \eqref{dotrho}. From \eqref{dotrho} we
now get $\dot{\theta}+k\theta\leq f$ in $\mathcal{D}'(0,\infty)$ and hence
$$\frac{d}{dt}\Big(e^{kt}\Big(\theta-\frac{l}{k}\Big)-\int_0^t e^{k\tau}f(\tau)d\tau\Big)\leq 0\quad\mbox{in }\mathcal{D}'(0,\infty).$$
Setting
$$\omega=e^{kt}(\theta-\frac{l}{k})-\int_0^t e^{k\tau}f(\tau)d\tau,$$
we therefore have
\begin{equation}
\dot{\omega}\leq 0\quad\mbox{in }\mathcal{D}'(0,\infty),
\label{dotomega}
\end{equation}
from which we now show that
\begin{equation}
\omega(t)\leq\omega(s),
\label{omegadecr}
\end{equation}
for a.e. $t$, $s\in(0,\infty)$, with $t\geq s$. Indeed, let $\{\chi_{\epsilon}\}_{\epsilon>0}$,
$\chi_{\epsilon}\geq 0$, be a sequence
of mollifiers belonging to $\mathcal{D}(\mathbb{R})$
and consider the convolution $\omega_{\epsilon}=\chi_{\epsilon}\ast\overline{\omega}$,
where $\overline{\omega}$ is the trivial extention of $\omega$ to the whole real line.
Since $\omega_{\epsilon}\in C^{\infty}(\mathbb{R})$, we have, for every $\varphi\in\mathcal{D}(0,\infty)$,
$\varphi\geq 0$
\begin{align}
&\int_0^{\infty}\dot{\omega_{\epsilon}}\varphi=-\int_0^{\infty}\omega_{\epsilon}\dot{\varphi}
=-\int_0^{\infty}\dot{\varphi}(t)dt\int_{\mathbb{R}}\chi_{\epsilon}(t-\tau)\overline{\omega}(\tau)d\tau\nonumber\\
&=-\int_{\mathbb{R}}\overline{\omega}(\tau)d\tau\int_{\mathbb{R}}\chi_{\epsilon}(t-\tau)\dot{\varphi}(t)dt
=-\int_{\mathbb{R}}\overline{\omega}(\tau)(\chi_{\epsilon}\ast\dot{\varphi})(\tau)d\tau\nonumber\\
&
=-\int_0^{\infty}\omega(\tau)\frac{d}{d\tau}(\chi_{\epsilon}\ast\varphi)(\tau)d\tau\leq 0
\end{align}
for $\epsilon>0$ small enough (i.e., such that
$\chi_{\epsilon}\ast\varphi\in\mathcal{D}(0,\infty)$, that occurs
when $\epsilon<\min(\mbox{supp}\varphi)$), due to
\eqref{dotomega}. Hence $\dot{\omega_{\epsilon}}(\tau)\leq 0$
for every $\tau\in(0,\infty)$, from which we deduce
$\omega_{\epsilon}(t)\leq\omega_{\epsilon}(s)$, for every $t$,
$s\in(0,\infty)$, with $t\geq s$. Letting $\epsilon\to 0$ and using
the fact that $\omega_{\epsilon}\to\omega$ a.e. in $(0,\infty)$, we
get \eqref{omegadecr}. Thus, on account of the definition of
$\omega$, from \eqref{omegadecr} we deduce \eqref{thesis1}.

Suppose now that $\theta:[0,\infty)\to\mathbb{R}$ is a l.s.c. representative
and that \eqref{intinq} holds
for a.e. $t,s\in(0,\infty)$ with $t\geq s$.
Let $N_1$ be a null set such that \eqref{intinq} holds for every $t$, $s\in(0,\infty)-N_1$,
with $t\geq s$. Let $t\in[0,\infty)$, $s\in(0,\infty)-N_1$ and take a sequence $t_j\in(0,\infty)-N_1$
such that $t_j\to t$. Write \eqref{intinq} for $s$ and $t_j$.
By virtue of the lower semicontinuity of $\theta$
we see that \eqref{intinq} holds also for all $t\geq s$ and a.e $s\in(0,\infty)$.
The same argument can be applied to \eqref{thesis1}.
Suppose in addition that the l.s.c. representative $\theta$ satisfies \eqref{intinq} also for $s=0$ and for
all $t\in[0,\infty)$.
Take a sequence $t_j\in(0,\infty)-N_1$ such that
$t_j\to 0$ and write \eqref{intinq} for $s=0$ and $t=t_j$.
By virtue of the lower semicontinuity of $\theta$ we get $\theta(t_j)\to\theta(0)$.
Now, let $N_2$ be a null set such that \eqref{thesis1} holds for
for every $s\in(0,\infty)-N_2$ and every $t\geq s$ and take a sequence $s_k\in(0,\infty)-N$, where $N=N_1\cup N_2$, such that
$s_k\to 0$. Write \eqref{thesis1} for $s=s_k$ and for $t\in (0,\infty)$.
Since $\theta(s_k)\to\theta(0)$, by letting $k\to\infty$ in \eqref{thesis1} we get \eqref{thesis2}.

Finally, suppose that $f$ has the form $f(t)=l+g(t)$, with $l\in\mathbb{R}$ a given constant
and $g$ translation bounded in $L^1_{loc}([0,\infty))$. By observing that (see, e.g., \cite[Chap.~XV, Cor.~1.7]{CV})
$$\int_0^t e^{-k(t-\tau)}g(\tau)d\tau\leq\frac{\|g\|_{L^1_{tb}}}{1-e^{-k}},$$
we immediately get \eqref{trbd}.
\end{proof}

Henceforth we shall denote by $u\in C_w([0,\infty);G_{div})$ and $\varphi\in
C_w([0,\infty);H)$ weakly continuous representatives of $u$ and
$\varphi$, where $[u,\varphi]=:z$ is the weak solution corresponding
to $u_0$ and $\varphi_0$ given by Theorem \ref{exist}.

The following lemma, which will be used to prove the dissipative
estimate in 3D, ensures the lower semicontinuity of the energy
$\mathcal{E}(z(\cdot))$ from $[0,\infty)$ to $\mathbb{R}$.
\begin{lem}
Let $z:=[u,\varphi]$ be the weak solution corresponding to $u_0$ and $\varphi_0$ and given
by Theorem \ref{exist}.
Then, the
function $\mathcal{E}(z(\cdot)):[0,\infty)\to\mathbb{R}$ is lower
semicontinuous.
\label{templsc}
\end{lem}
\begin{proof}
Let us represent the potential $F$ as
\begin{equation}
F(s)=\widetilde{G}(x,s)-a(x)\frac{s^2}{2},\quad\forall s\in\mathbb{R},\quad\mbox{a.e. }x\in\Omega,
\label{potrepr}
\end{equation}
where $\widetilde{G}(x,\cdot)$ is strictly convex for a.e.
$x\in\Omega$, owing to (A3). Then, the energy
$\mathcal{E}(z(\cdot))$ takes the form
$$\mathcal{E}(z(t))=\frac{1}{2}\|u(t)\|^2-\frac{1}{2}(\varphi(t),J\ast\varphi(t))
+\int_{\Omega}\widetilde{G}(x,\varphi(x,t))dx.$$
Therefore, the lower semicontinuity of $\mathcal{E}(z(\cdot)):[0,\infty)\to\mathbb{R}$
is a consequence of the weak lower semicontinuity in $G_{div}$ of the $L^2-$norm,
of the compactness of the convolution operator $J\ast\cdot:H\to H$ and
of the convexity of the integral funcional in $\widetilde{G}$ given by the last term
in the relation above.
\end{proof}

In dimension three, we can prove that the same global weak solution
constructed in Theorem \ref{exist} also satisfies energy inequality \eqref{ei} between two arbitrary times $s$ an $t$ (i.e., for a.e. $s\geq 0$, including $s=0$ and for all $t\geq s$), provided that a further growth assumption on $F$ is fulfilled (not needed in dimension two). This is stated in the following

\begin{lem}
Assume (A1)-(A4) hold. In addition, suppose that (A5) holds with $p\in(6/5,2]$ when $d=3$.
Let $z:=[u,\varphi]$ be the weak solution (in the sense of Definition \ref{wfdfn})
corresponding to $u_0$ and $\varphi_0$ and given
by Theorem \ref{exist}. Then, the following
energy inequality is satisfied
\begin{align}
& \mathcal{E}(z(t))+\int_s^t(2\|\sqrt{\nu(\varphi)}Du\|^2+\|\nabla\mu\|^2)d\tau\leq\mathcal{E}(z(s))+\int_s^t\langle h(\tau),
u(\tau)\rangle d\tau,
\label{eist}
\end{align}
for a.e. $s\geq 0$, including $s=0$, and for every $t\geq s$,
where $\mu=a\varphi-J\ast\varphi+F'(\varphi)$.
\end{lem}
\begin{proof}
We can argue as in the proof of \eqref{ei} (see \cite[Theorem
1]{CFG}) and integrate the energy identity satisfied by the
approximate solutions $z_n:=[u_n,\varphi_n]$ of the Faedo-Galerkin
scheme between $s$ and $t$, with $0\leq s\leq t$. When we pass to
the limit as $n\to\infty$ in the integrated identity we have to consider
the functional integral term $\int_{\Omega}F(\varphi_n(s))$ on the
right hand side. Recalling now the bounds for the sequences
$\{u_n\}$, $\{\varphi_n\}$ and $\{\varphi_n'\}$, in particular
(see \cite{CFG})
$$\|\varphi_n\|_{L^2(0,T;V)}\leq c,\qquad\|\varphi_n'\|_{L^{4/3}(0,T;V')}\leq c,\qquad\forall T>0,$$
and using the Aubin-Lions lemma which ensures the compact
embedding
$$L^2(0,T;V)\cap W^{1,4/3}(0,T;V')\hookrightarrow\hookrightarrow L^2(0,T;L^{p'}(\Omega)),$$
with $p'\in [2,6)$ (since $p\in (6/5,2]$), at least for a subsequence
we have
$$\varphi_n(s)\to\varphi(s),\qquad\mbox{strongly in }L^{p'}(\Omega),$$
for a.e. $s>0$.
Since $F$ has a polynomial growth of order $p'$ (cf. Remark 1), then by Lebesgue's
theorem we have
$$\int_{\Omega}F(\varphi_n(s))\to\int_{\Omega}F(\varphi(s)),$$
for a.e. $s>0$. Using now the lower semicontinuity of the norm we therefore
get \eqref{eist} for a.e $s$ and a.e. $t$, with $0\leq s\leq t$.
By means of a suitable approximation of the initial datum $\varphi_0$ and of the
fact that $F$ is a quadratic perturbation of a convex function we deduce,
as in the proof of \cite[Theorem 1]{CFG}, that \eqref{eist} holds also for $s=0$
and for a.e $t>0$. Finally, due to the lower semicontinuity of $\mathcal{E}(z(\cdot)):[0,\infty)\to\mathbb{R}$
(see Lemma \ref{templsc}), we deduce that \eqref{eist} holds also for every $t\geq s$.
\end{proof}

\begin{oss}
{\upshape
If the growth restriction on $F$ does not hold, then
we can only say that for every $s\geq 0$ there exists a global weak solution (with initial data given at time $s$
by the solution constructed in Theorem \ref{exist} with initial data given at time $s=0$ and considered at time $s$)
satisfying \eqref{eist} for all $t\geq s$ (such global weak solution not necessarily coincides, between $s$ and $t$, with
the global weak solution constructed in Theorem \ref{exist} with initial data given at time $s=0$ and generally depends on $s$).
}
\end{oss}

We can now prove the following

\begin{thm}
\label{energyineq}
Suppose (A1)-(A3) and (A5)-(A6) hold. Also, let $h\in
L^2_{tb}(0,\infty,V_{div}')$ be given. Then every weak solution
$z=[u,\varphi]$ (in the sense of Definition \ref{wfdfn2}) fulfilling the
energy inequality \eqref{eist} for a.e. $s\geq 0$, including $s=0$, and every $t\geq s$,
satisfies the dissipative inequality
\begin{equation}
\mathcal{E}(z(t))\leq \mathcal{E}(z_0)e^{-kt}+F(m_0)|\Omega|+K,
\label{dissest3D}
\end{equation}
for all $t\geq 0$, where $m_0=(\varphi_0,1)$, and $k$, $K$ are two
positive constants that are independent of the initial data with $K$
depending on $\Omega$, $\nu_1$, $J$, $F$ and on
$\|h\|_{L^2_{tb}(0,\infty;V_{div}')}$.
\end{thm}

\begin{oss}
{\upshape Since, under the growth restriction $p\in(6/5,2]$ the weak solution of Theorem \ref{exist}, which is
constructed via a Faedo-Galerkin method, satisfies the energy
inequality \eqref{eist}, then for such weak solution the dissipative
estimate \eqref{dissest3D} holds. Nevertheless, the validity of
\eqref{dissest3D} does not depend neither on the fact that the weak solution
is constructed as in Theorem \ref{exist} nor on the growth restriction, but it
relies on the validity of the energy inequality \eqref{eist} only.}
\end{oss}

\begin{proof}
Let us first suppose that $(\varphi_0,1)=0$ and
multiply equation $\mu=a\varphi-J\ast\varphi+F'(\varphi)$ by $\varphi$ in $L^2(\Omega)$.
We obtain
\begin{equation}
(\mu,\varphi)=\frac{1}{2}
\int_{\Omega}\int_{\Omega}J(x-y)(\varphi(x)-\varphi(y))^2 dxdy+(F'(\varphi),\varphi).\label{muphi}
\end{equation}
Observe now that, by writing the potential $F$ as in \eqref{potrepr} and using the convexity
of $\widetilde{G}(x,\cdot)$,
then, for every $s\in\mathbb{R}$ and a.e. $x\in\Omega$ we have
$$(F'(s)+a(x)s)s\geq F(s)+\frac{a(x)}{2}s^2-F(0),$$
and hence
$$F'(s)s\geq F(s)-\frac{a(x)}{2}s^2-F(0).$$
Thus, from \eqref{muphi} we get
\begin{align}
(\mu,\varphi)\geq\frac{1}{2}
\int_{\Omega}\int_{\Omega}J(x-y)(\varphi(x)-\varphi(y))^2 dxdy
+\int_{\Omega}F(\varphi(t))-\frac{1}{2}\|\sqrt{a}\varphi\|^2-F(0)|\Omega|.
\label{muphi2}
\end{align}
On the other hand, we have
$$(\mu,\varphi)=(\mu-\overline{\mu},\varphi)\leq C_P\|\nabla\mu\|\|\varphi\|,$$
where $C_P$ is the Poincar\'{e}-Wirtinger constant and
$\overline{\mu}:=\frac{1}{|\Omega|}(\mu,1)$.
Furthermore, (A6) implies that there exist $C_9>0$ and $C_{10}>0$ such
that $F(s)\geq C_9|s|^{2+2q}-C_{10}$ for all $s\in\mathbb{R}$, and therefore from \eqref{muphi2} we get
\begin{align}
&\frac{1}{8}\int_{\Omega}\int_{\Omega}J(x-y)(\varphi(x)-\varphi(y))^2 dxdy+\frac{1}{2}\int_{\Omega}F(\varphi)
+\frac{C_9}{2}\int_{\Omega}|\varphi|^{2+2q}-\frac{C_{10}}{2}|\Omega|\nonumber\\
&\leq C_{11}\|\varphi\|^2+\|\nabla\mu\|^2+F(0)|\Omega|,\nonumber
\end{align}
with $C_{11}=\frac{1}{4}(3\|J\|_{L^1}+C_{P}^2)$.
We deduce
\begin{equation}
\frac{1}{8}\int_{\Omega}\int_{\Omega}J(x-y)(\varphi(x)-\varphi(y))^2 dxdy+\frac{1}{2}\int_{\Omega}F(\varphi)
\leq\|\nabla\mu\|^2+C_{12},
\label{jr}
\end{equation}
and therefore
\begin{equation}
\frac{1}{2}\mathcal{E}(z(t))\leq C_{13}\Big(\frac{\nu_1}{2}\|\nabla u(t)\|^2+\|\nabla\mu(t)\|^2\Big)+C_{12},
\label{hr}
\end{equation}
where $C_{13}=\max(1,1/2\lambda_1\nu_1)$, $\lambda_1$ being the first eigenvalue of the Stokes operator $A$.
We point out that all constants only depend on the parameters
of the problem and are independent of the initial data.

Observe that the energy inequality \eqref{eist} yields
\begin{equation*}
\mathcal{E}(z(t))+\int_s^t\Big(\frac{\nu_1}{2}\|\nabla u\|^2
+\|\nabla\mu\|^2\Big)d\tau
\leq\mathcal{E}(z(s))+\frac{1}{2\nu_1}\int_s^t\|h(\tau)\|_{V_{div}'}^2
d\tau,
\end{equation*}
for a.e $s\geq 0$, including $s=0$, and every $t\geq s$.
Therefore, on account of \eqref{hr}, we obtain the integral
inequality
\begin{equation}
 \mathcal{E}(z(t))+k\int_0^t\mathcal{E}(z(\tau))
 d\tau\leq l(t-s)+\frac{1}{2\nu_1}\int_s^t\|h(\tau)\|_{V_{div}'}^2d\tau
 +\mathcal{E}(z(s))
 +k\int_0^s\mathcal{E}(z(\tau))d\tau,
\label{eiconseq}
 \end{equation}
for a.e. $s\geq 0$, including $s=0$, and every $t\geq s$,
where $k=1/2C_{13}$ and $l=C_{12}/C_{13}$.
Since, by Lemma \ref{templsc}, $\mathcal{E}(z(\cdot)):[0,\infty)\to\mathbb{R}$ is
lower semicontinuous, then by applying Lemma \ref{Ball} we deduce that
\begin{equation}
\mathcal{E}(z(t))\leq\mathcal{E}(z_0)e^{-kt}+K,
\label{dissest-0}
\end{equation}
for all $t\geq 0$, where the constant $K>0$ is given by
\begin{equation}
K=\frac{1}{2\nu_1}
\frac{1}{1-e^{-k}}\|h\|_{L^2_{tb}(0,\infty;V_{div}')}^2+\frac{l}{k}.
\label{constK}
\end{equation}

Now, let $z:=[u,\varphi]$ be a weak solution corresponding to data $z_0:=[u_0,\varphi_0]$
with $m_0:=(\varphi_0,1)\not=0$ for the problem with potential $F$
and fulfilling the energy inequality \eqref{eist} for a.e. $s\geq 0$, including $s=0$, and for every $t\geq s$.
Then
$\widetilde{z}:=[u,\widetilde{\varphi}]$, where $\widetilde{\varphi}=\varphi-m_0$,
is a weak solution with data $\widetilde{z}_0:=[u_0,\varphi_0-m_0]$ for the same problem with
potential $\widetilde{F}$ and viscosity $\widetilde{\nu}$ given by
$$\widetilde{F}(s):=F(s+m_0)-F(m_0),\qquad\widetilde{\nu}(s):=\nu(s+m_0).$$
The weak solution $\widetilde{z}$ fulfills $(\widetilde{\varphi},1)=0$ and it can be easily checked that \eqref{eist} holds for $\widetilde{z}$, namely that we have
\begin{align}
& &\widetilde{\mathcal{E}}(\widetilde{z}(t))+\int_s^t(2\|\sqrt{\widetilde{\nu}(\widetilde{\varphi})}Du\|^2
+\|\nabla\widetilde{\mu}\|^2)d\tau\leq\widetilde{\mathcal{E}}(\widetilde{z}(s))+\int_s^t\langle h(\tau),
u\rangle d\tau,
\label{eisttilde}
\end{align}
for a.e. $s\geq 0$, including $s=0$, and for every $t\geq s$,
where
$$\widetilde{\mathcal{E}}(\widetilde{z}(t)):
=\frac{1}{2}\|u(t)\|^2+\frac{1}{4}
\int_{\Omega}\int_{\Omega}J(x-y)(\widetilde{\varphi}(x,t)-\widetilde{\varphi}(y,t))^2
dxdy+\int_{\Omega}\widetilde{F}(\widetilde{\varphi}(t))$$
and $\widetilde{\mu}:=a\widetilde{\varphi}-J\ast\widetilde{\varphi}+\widetilde{F}'(\widetilde{\varphi})
=a\varphi-J\ast\varphi+F'(\varphi)=\mu$.
Indeed \eqref{eisttilde} is an immediate consequence of the following identity
\begin{equation}
\widetilde{\mathcal{E}}(\widetilde{z}(t))=\mathcal{E}(z(t))-F(m_0)|\Omega|,
\label{mathcalE}
\end{equation}
and of the fact that $z$ satisfies \eqref{eist}.
By applying the argument above we therefore deduce that the weak solution $\widetilde{z}$
satisfies \eqref{dissest-0} and by combining this inequality with \eqref{mathcalE} we get \eqref{dissest3D}.
\end{proof}

\begin{oss}
{\upshape Assumption (A6) in Theorem \ref{energyineq} can be
replaced by (A4) provided that either (i)
$c_1>\frac{3}{2}\|J\|_{L^1}$ or (ii) $C_P<\frac{c_0}{2\|\nabla
J\|_{L^1}}$ holds. Indeed, using (A4) from \eqref{muphi2} we have
\begin{align}
&\frac{1}{8}\int_{\Omega}\int_{\Omega}J(x-y)(\varphi(x)-\varphi(y))^2 dxdy+\frac{1}{2}\int_{\Omega}F(\varphi)
+\frac{c_1}{2}\|\varphi\|^2-\frac{c_2}{2}|\Omega|\nonumber\\
&\leq\frac{3}{4}(\varphi,J\ast\varphi)+C_P\|\nabla\mu\|\|\varphi\|\leq\frac{3}{4}\|J\|_{L^1}\|\varphi\|^2
+C_P\|\nabla\mu\|\|\varphi\|.
\label{A4forA6}
\end{align}
From \eqref{A4forA6} it clear that, if (i) holds an inequality like
\eqref{jr} can be obtained again. On the other hand, if (ii) holds then
it can be proved that (see \cite[Remark 9]{CFG})
$\|\nabla\mu\|^2\geq\beta\|\nabla\varphi\|^2$ for
$\overline{\varphi}=0$, where $\beta=(c_0-2C_P\|\nabla
J\|_{L^1})^2$. Then an inequality like \eqref{jr} can still be
recovered  from \eqref{A4forA6}. Therefore, inequality \eqref{hr}
(with different values of $C_{12},C_{13}$) holds in both cases and
the same argument used in the proof can be used to show that every
weak solution $z=[u,\varphi]$ (in the sense of Definition
\ref{wfdfn}) which fulfills the energy inequality \eqref{eist} also
satisfies \eqref{dissest3D}.}
\end{oss}


\section{Existence of a trajectory attractor}\setcounter{equation}{0}
\label{sec6}
There exist various methods to define a generalized notion of global attractor for the Navier-Stokes equations
in dimension three. Here we will follow the so-called trajectory approach presented in \cite{CV} (see also \cite{CV2,FS,Se}). For alternative approaches, the reader is referred to, e.g., \cite{CF,Cu,KaVa,Ro} and references therein. In this section our assumption (A6) will be slightly strengthened.
We shall deal mainly with the case $d=3$, though the case $d=2$ will also be considered.

We begin to define, for any given $m\geq 0$ and $M>0$, the functional space
\begin{align*}
\mathcal{F}_M =\Big\{&[v,\psi] \in L^\infty(0,M;G_{div}\times L^{p^\prime}(\Omega))\cap L^2(0,M;V_{div}\times V)\,:\\
&v_t \in L^{4/3}(0,M;V^\prime_{div}),\; \psi_t \in L^2(0,M;V^\prime), \; |(\psi(t),1)|\leq m, \; t\in [0,M]\Big\}.
\end{align*}
which is a complete metric space with respect to the metric induced by the norm
\begin{align*}
\Vert[v,\psi]\Vert_{\mathcal{F}_M} &= \Vert [v, \psi]\Vert_{L^\infty(0,M;G_{div}\times L^{p^\prime}(\Omega))}
+ \Vert [\nabla v,\nabla \psi]\Vert_{L^2(0,M; H \times H)}\\
&+ \Vert v_t \Vert_{L^{4/3}(0,M;V^\prime_{div})} + \Vert\psi_t \Vert_{L^2(0,M;V^\prime)}.
\end{align*}
Then we introduce the spaces
\begin{align*}
\mathcal{F}^+_{loc} =\Big\{&[v,\psi] \in L^\infty_{loc}([0,\infty);G_{div}\times L^{p^\prime}(\Omega))\cap L^2_{loc}([0,\infty); V_{div}\times V)\,:\\
&v_t \in L^{4/3}_{loc}([0,\infty);V^\prime_{div}),\; \psi_t \in L^2_{loc}([0,\infty);V^\prime),
\; |(\psi(t),1)|\leq m, \; t\geq 0\Big\},\\
 \mathcal{F}^+_{b} =\Big\{&[v,\psi] \in L^\infty(0,\infty;G_{div}\times L^{p^\prime}(\Omega))\cap L^2_{tb}(0,\infty; V_{div}\times V)\,:\\
&v_t \in L^{4/3}_{tb}(0,\infty;V^\prime_{div}),\; \psi_t \in L^2_{tb}(0,\infty;V^\prime),
\; |(\psi(t),1)|\leq m, \; t\geq 0\Big\}.
\end{align*}
We recall that $\mathcal{F}^+_{b}$ can be viewed as a complete metric space as $\mathcal{F}_M$ by endowing it with the metric induced by the norm
\begin{align*}
\Vert[v,\psi]\Vert_{\mathcal{F}^+_{b}} &= \Vert [v, \psi]\Vert_{L^\infty(0,\infty;G_{div}\times L^{p^\prime}(\Omega))}
+ \Vert [\nabla v,\nabla \psi]\Vert_{L^2_{tb}(0,\infty; H \times H)}\\
&+ \Vert v_t \Vert_{L^{4/3}_{tb}(0,\infty;V^\prime_{div})} + \Vert\psi_t \Vert_{L^2_{tb}(0,\infty;V^\prime)}.
\end{align*}

We will indicate by $\Theta_M$ the space $\mathcal{F}_M$ endowed with the following sequential topology

\begin{defn}
\label{topconv}
$\{[v_n,\psi_n]\}\subset \mathcal{F}_M$ converges to $[v,\psi]\in\mathcal{F}_M$
as $n \to \infty$ in $\Theta_M$ if
\begin{align*}
&v_n\rightharpoonup v\quad
\mbox{weakly}^{\ast}\mbox{ in }L^{\infty}(0,M;G_{div})
\;\mbox{and weakly in }L^2(0,M;V_{div}),\\
&(v_n)_t\rightharpoonup v_t\quad\mbox{weakly in }L^{4/3}(0,M;V_{div}'),\\
&\psi_n\rightharpoonup\psi\quad
\mbox{weakly}^{\ast}\mbox{ in }L^{\infty}(0,M;L^{p^\prime}(\Omega))
\;\mbox{and weakly in }L^2(0,M;V),\\
&(\psi_n)_t\rightharpoonup\psi_t\quad\mbox{weakly in }L^2(0,M;V').
\end{align*}
\end{defn}
Then the inductive limit of $\{\Theta_M\}_{M>0}$ will be denoted by $\Theta^+_{loc}$ (see \cite[Chap.~XII, Def.~1.3]{CV}).
We recall that $\Theta_M$ and $\Theta^+_{loc}$ have countable topological bases.

\begin{oss}
{\upshape
We have given all the definitions above with reference to the case $d=3$. If $d=2$,
in the definitions of functional spaces $\mathcal{F}_M,\mathcal{F}^+_{loc},\mathcal{F}^+_b$
(and in the corresponding norms)
we replace the regularity assumptions $L^{4/3},L^{4/3}_{loc},L^{4/3}_{tb}$
for $v_t$ with $L^2,L^2_{loc},L^2_{tb}$, respectively, while
in Definition \ref{topconv} the second condition is replaced by the convergence
$(v_n)_t\rightharpoonup v_t$ weakly in $L^2(0,M;V_{div}')$.
}
\end{oss}

We now consider the union of all weak solutions with external force $h$ (in the sense of
Definition \ref{wfdfn2} with $T=M$) satisfying inequality
\eqref{eist} on $[0,M]$ and we denote it by $\mathcal{K}^M_h$,
while $\mathcal{K}^+_h$ will be the union of all weak solutions in
$\mathcal{F}^+_{loc}$ with external force $h$ satisfying $\eqref{eist}$ on $[0,\infty)$.

The first result concerns with
the $(\Theta_M,L^2(0,M;V_{div}'))$-closure of the family
$\{\mathcal{K}^M_h,h\in L^2(0,M;V_{div}')\}$.
More precisely, we prove that the graph set
$$\bigcup_{h\in \mathcal{L}_M}\mathcal{K}^M_h\times\{h\}$$
is closed in the topological space $\Theta_M\times\mathcal{L}_M$, where $\mathcal{L}_M$ is $L^2(0,M;V_{div}')$
or $L^2_w(0,M;G_{div})$
if $d=3$, and $\mathcal{L}_M$ is $L^2_w(0,M;V_{div}')$ if $d=2$
(cf., for instance, \cite[Chap.~XV, Prop.~1.1]{CV}).


\begin{prop}
\label{traj1}
Let (A1)-(A3) hold. In addition, suppose that (A5) holds with  $p\in (\frac65,\frac32]$ if $d=3$ and
with $p\in(1,2)$ if $d=2$ and that
(A6) holds with $2q+2=p^\prime$. Let $h_m \in L^2(0,M,V^\prime_{div})$ and consider
$[v_m,\psi_m]\in \mathcal{K}^M_{h_m}$. Let $\{[v_m,\psi_m]\}$ converge to $[v,\psi]$
according to Definition \ref{topconv}. If
\begin{align*}
& h_m \rightharpoonup h\quad
\mbox{weakly} \mbox{ in }L^2(0,M;V^\prime_{div}), \quad d=2,\\
& h_m \to h\quad
\mbox{strongly} \mbox{ in }L^2(0,M;V^\prime_{div})
\mbox{ or }h_m \rightharpoonup h\mbox{ weakly in } L^2(0,M;G_{div}), \quad d=3,
\end{align*}
then $[v,\psi]\in \mathcal{K}^M_{h}$.
\end{prop}

\begin{proof}
Consider first the case $d=3$.
Since $[v_m,\psi_m]\in \mathcal{K}^M_{h_m}$, each weak solution $z_m:=[v_m,\psi_m]$ is such that:
(i) $v_m\in L^{\infty}(0,M;G_{div})\cap L^2(0,M;V_{div})$,
$(v_m)_t\in L^{4/3}(0,M;V_{div}')$,
$\psi_m\in L^{\infty}(0,M;L^{2+2q}(\Omega))\cap L^2(0,M;V)$,
 $(\psi_m)_t\in L^2(0,M;V')$ (we are assuming $q\geq 1/2$),
$\mu_m\in L^2(0,M;V)$, where $\mu_m=a\psi_m-J\ast\psi_m+F'(\psi_m)$;
(ii) the weak formulation \eqref{wf1}, \eqref{wf2}, \eqref{chempot} for $[v_m,\psi_m]$
and $\mu_m$ holds with external force $h_m$, and (iii) the energy inequality
\begin{align}
&\mathcal{E}(z_m(t))+\int_s^t(2\|\sqrt{\nu(\psi_m)}Dv_m\|^2+\|\nabla\mu_m\|^2)d\tau\leq\mathcal{E}(z_m(s))
+\int_s^t\langle h_m(\tau),
u_m\rangle d\tau
\label{eim}
\end{align}
is satisfied for every $m\in\mathbb{N}$, for a.e. $s\in[0,M]$, including $s=0$, and for every
$t\in[0,M]$ with $t\geq s$.
Observe that the third convergence condition in Definition 3 is compatible with the regularity
property $\psi_m\in L^{\infty}(0,M;L^{2+2q}(\Omega))$ due to the constraint $2q+2=p'$.
Note that $q\geq 1/2$ since $p\leq 3/2$.

Thanks to the convergences listed in Definition \ref{topconv} and to the
polynomial control \eqref{growth}
on $F$ it is easy to see that there exists $c>0$ such that
\begin{equation}
\label{encontrol}
|\mathcal{E}(z_m(s))|\leq c,\qquad\forall m,\quad\mbox{a.e. }s\in(0,M)
\end{equation}
Hence, \eqref{eim} and the convergence assumption on the sequence $\{h_m\}$ entail
the control $\Vert \nabla\mu_m\Vert_{L^2(0,M;H)}\leq c$. Furthermore, since
$(\mu_m,1)=(F'(\psi_m),1)$ and since from (A4) we have
$|F'(\psi_m)|^p\leq c|F(\psi_m)|+c\leq c|\psi_m|^{p'}+c$ which, along with
the third convergence of Definition \ref{topconv}, implies that
\begin{equation}
\Vert F'(\psi_m)\Vert_{L^{\infty}(0,M;L^p(\Omega))}\leq c,
\end{equation}
we deduce that $|(\mu_m,1)|\leq c$ and therefore that $\Vert \mu_m \Vert_{L^2(0,M;V)}\leq c$.
Observe that we also have the estimate $\|F'(\psi_m)\|_{L^2(0,M;V)}\leq c$.
As a consequence, there exists $\mu\in L^2(0,M;V)$ such that for a subsequence we have
\begin{equation}
\mu_m\rightharpoonup\mu,\qquad\mbox{weakly in }L^2(0,M;V).
\label{convmu}
\end{equation}
Definition \ref{topconv} also implies that, up to subsequences, $\{[v_m,\psi_m]\}$
strongly converges to $[v,\psi]$ in $L^2(0,M;G_{div} \times H)$ and thus
$\{\psi_m\}$ also converges to $\psi$ almost everywhere in
$\Omega\times(0,M)$. We hence get that $\mu=a\psi-J\ast\psi+F'(\psi)$.
Now, on account of Definition \ref{topconv}, of the strong convergences
obtained above and of \eqref{convmu}, we can now pass to the limit in the variational
formulation \eqref{wf1}, \eqref{wf2}, \eqref{chempot} for the weak solution $[v_m,\psi_m]$
with external force $h_m$
and thus deduce that $[v,\psi]$ is a weak solution with external force $h$.

It remains to prove that the weak solution $[v,\psi]$ satisfies the energy inequality \eqref{eist}
on $[0,M]$ with external force $h$. To this aim we pass to the limit in \eqref{eim} as $m\to\infty$.
We exploit the convergence $h_m\to h$, strongly in
$L^2(0,M;V_{div}')$ and, in order to pass to the limit in the nonlinear functional term $\int_{\Omega}F(\psi_m(s))$
on the right hand side of \eqref{eim} we notice that, due to the third and fourth convergences assumed in Definition \ref{topconv} and to \eqref{Aubin-Lions}, we have that $\psi_n(s)\to\psi(s)$
strongly in $L^{p'}(\Omega)$ for a.e. $s>0$ and hence, since $p>6/5$,
we get $\int_{\Omega}F(\psi_m(s))\to\int_{\Omega}F(\psi(s))$ for a.e. $s>0$.
By also using \eqref{convmu}, the convergence
$$\sqrt{\nu(\psi_m)}Dv_m\rightharpoonup\sqrt{\nu(\psi)}Dv\qquad\mbox{weakly in }L^2(0,M;H),$$
(cf. \eqref{weakvisc}) and the lower semicontinuity of the $L^2(0,M;H)-$norm we thus get
that $[v,\psi]$ with external force $h$ satisfies \eqref{eist}
for a.e. $s\in[0,M]$, including $s=0$, and for every
$t\in[0,M]$ with $t\geq s$.
Hence $[v,\psi]\in\mathcal{K}^M_h$. The same conclusion holds if we suppose that $h_m\rightharpoonup h$ weakly in $L^2(0,M;G_{div})$. Indeed, arguing as in \cite[Chap.~XV, Prop.~1.1]{CV} and relying on the strong convergence
$u_m\to u$ in $L^2(0,M;G_{div})$ we have that
$$\int_s^t\langle h_m(\tau),u_m(\tau)\rangle d\tau\to
\int_s^t\langle h(\tau),u(\tau)\rangle d\tau,\qquad\mbox{as }m\to\infty.$$
If $d=2$, the situation is
easier since the energy identity can be deduced from the weak
formulation (see also \cite[Chap.~XV, proof of Prop.~1.1]{CV}).
\end{proof}

\begin{oss}
{\upshape
The main reason for assuming that
$\psi_n\rightharpoonup\psi$ weakly$^{\ast}$ in $L^{\infty}(0,M;L^{p^\prime}(\Omega))$
in Definition 3, rather then the apparently more natural convergence condition
$\psi_n\rightharpoonup\psi$ weakly$^{\ast}$ in $L^{\infty}(0,M;L^{2+2q}(\Omega))$,
is in order to ensure \eqref{encontrol}. Obviously, as pointed out above, the relation $p'=2+2q$
is needed.
}
\end{oss}

Consider now $h_0\in L^2_{tb}(0,\infty;V^\prime_{div})$ so that $h_0$ is translation
compact in $L^2_{loc,w}([0,\infty);V^\prime_{div})$. Then set
$$
\mathcal{H}_+(h_0) := \left[\left\{ h_0(\cdot+\tau)\,\vert\, \tau\geq 0\right\}\right]_{L^2_{loc,w}([0,\infty);V^\prime_{div})},
$$
where $[\cdot]_X$ denotes the closure in the space $X$.
The following property will be useful in the next proposition:
if $h_0\in L^2_{tb}(0,\infty;V_{div}')$ and $h\in\mathcal{H}_+(h_0)$, then $h\in L^2_{tb}(0,\infty;V_{div}')$ as well and
\begin{equation}
\|h\|_{L^2_{tb}(0,\infty;V_{div}')}\leq\|h_0\|_{L^2_{tb}(0,\infty;V_{div}')}.
\label{prophull}
\end{equation}
We recall that the translation semigroup $T(t)$ is continuous on $\mathcal{H}_+(h_0)$ and $T(t)\mathcal{H}_+(h_0)=\mathcal{H}_+(h_0)$
for all $t\geq 0$. This translation semigroup can also be defined on $\mathcal{K}^+_{h}$
for any $h\in \mathcal{H}_+(h_0)$. Indeed, if $[v,\psi]\in \mathcal{K}^M_{h}$ then $T(\tau)[v,\psi]\in
\mathcal{K}^{M-\tau}_{T(\tau)h}$, i.e.
$T(\tau)\mathcal{K}^M_{h} \subseteq \mathcal{K}^{M-\tau}_{T(\tau)h}$ for all $M\geq\tau\geq 0$.
Thus, recalling \cite[Chap.~XIV, Props.~1.1 and 1.2]{CV}, we have, for all $t\geq 0$,
\begin{equation*}
T(t)\mathcal{K}^+_{h} \subseteq \mathcal{K}^+_{T(t)h},  \quad
T(t)\mathcal{K}^+_{\mathcal{H}_+(h_0)}\subseteq \mathcal{K}^+_{\mathcal{H}_+(h_0)}.
\end{equation*}
where $\mathcal{K}^+_{\mathcal{H}_+(h_0)}:= \bigcup_{h\in \mathcal{H}_+(h_0)} \mathcal{K}^+_{h}$
is the so-called {\it united trajectory space} (see \cite[Chap.~XIV, Def.~1.2]{CV}).

We can now prove the following (see \cite[Chap.~XV, Prop.~1.2]{CV})

\begin{prop}
\label{traj2} Let (A1)-(A3) hold. In addition, suppose that (A5) holds with  $p\in (1,\frac32]$ if $d=3$
and with $p\in(1,2)$ if $d=2$ and that
(A6) holds with $2q=p^\prime -2$. If $h_0 \in L^2_{tb}(0,\infty;V^\prime_{div})$ then, for all $h\in \mathcal{H}_+(h_0)$, we have $\mathcal{K}^+_{h} \subset \mathcal{F}^+_{b}$ and the following dissipative estimate
holds
\begin{equation}
\label{trajest}
\Vert T(t)[v,\psi]\Vert_{\mathcal{F}^+_{b}} \leq \Lambda_0 \Vert [v,\psi]\Vert_{L^\infty(0,1;G_{div}\times L^{p'}(\Omega))}
e^{-\kappa t} + \Lambda_1,
\end{equation}
for all $t\geq 1$ and all $[v,\psi]\in  \mathcal{K}^+_{h}$. Here $\Lambda_0$, $\kappa$ and $\Lambda_1$ are
positive constants with $k=\min(1/2,\lambda_1\nu)$ and
$\Lambda_0$, $\Lambda_1$ depending on $\nu_1,\nu_2,\lambda_1,F,J,|\Omega|$,
with $\Lambda_1$ depending also on $\|h_0\|_{L^2_{tb}(0,\infty;V_{div}')}$.
\end{prop}

\begin{proof}
Take $[v,\psi]\in  \mathcal{K}^+_{h}$. Then, by definition $z:=[v,\psi]$ is a weak solution
corresponding to the external force $h$ satisfying \eqref{eist} and hence
\eqref{eiconseq} on $[0,\infty)$. By applying Lemma \ref{Ballgen} we get
\begin{equation}
\mathcal{E}(z(t))
\leq \mathcal{E}(z(s))e^{-k(t-s)}
+\frac{1}{2\nu_1}\int_s^t e^{-k(t-\tau)}\left(\|h(\tau)\|_{V_{div}'}^2 + 2\nu_1 l \right)d\tau,
\nonumber
\end{equation}
for a.e. $s\geq 0$, including $s=0$ and for every $t\geq s$.
The constants $k,l$ are the same as in the proof of Theorem \ref{energyineq}.
In particular we have $k=\min(1/2,\lambda_1\nu_1)$.
Thus, we deduce
\begin{equation}
\label{trajdiss1}
\mathcal{E}(z(t))
\leq e^k\sup_{s\in (0,1)}\mathcal{E}(z(s))e^{-kt}
+\frac{1}{2\nu_1}\int_0^t e^{-k(t-\tau)}\left(\|h(\tau)\|_{V_{div}'}^2 + 2\nu_1 l \right)d\tau,
\end{equation}
for all $t\geq 1$. Now, notice that, due to \eqref{growth}, (A5) and to the assumption $p'=2+2q$
there exist two constants $k_1,k_2>0$ depending on $F$ and $J$ such that
\begin{eqnarray}
& &k_1(\|\varphi(s)\|_{L^{p'}(\Omega)}^{p'}+\|u(s)\|^2-1)\leq\mathcal{E}(z(s))\leq k_2(\|\varphi(s)\|_{L^{p'}(\Omega)}^{p'}+\|u(s)\|^2+1).
\label{contrE}
\end{eqnarray}
Henceforth $c$ will stand for a positive constant, that may vary from line to line,
that depends on $\nu_1,\lambda_1,F,J$ and $|\Omega|$.
From \eqref{trajdiss1} we obtain
\begin{align}
&
\|\varphi(t)\|_{L^{p'}(\Omega)}^{p'}+\|u(t)\|^2\leq
c(\|\varphi\|_{L^{\infty}(0,1;L^{p'}(\Omega))}^{p'}+\|u\|_{L^{\infty}(0,1;G_{div})}^2)e^{-kt}\nonumber\\
&
+\frac{1}{2\nu_1}\int_0^t e^{-k(t-\tau)}\|h(\tau)\|_{V_{div}'}^2 d\tau + \frac{l}{k}+c,\qquad\forall t\geq 1,
\label{q1}
\end{align}
which immediately leads to
\begin{align}
 &
 \|T(t)\varphi\|_{L^{\infty}(0,\infty;L^{p'}(\Omega))}^{p'}+\|T(t)u\|_{L^{\infty}(0,\infty;G_{div})}^2
\leq c(\|\varphi\|_{L^{\infty}}^{p'}+\|u\|_{L^{\infty}}^2)e^{-kt}
+K+c,
\label{1term}
\end{align}
for all $t\geq 1$,
where we have set $\|\varphi\|_{L^{\infty}}:=\|\varphi\|_{L^{\infty}(0,1;L^{p'}(\Omega))}$
and $\|u\|_{L^{\infty}}:=\|u\|_{L^{\infty}(0,1;G_{div})}$. The constant
$K$ also depends on $h_0$ and has the following form (see \eqref{constK} and \eqref{prophull})
$$K=\frac{1}{2\nu_1}\frac{1}{1-e^{-k}}\|h_0\|_{L^2_{tb}(0,\infty;V_{div}')}^2+\frac{l}{k}.$$
By \eqref{eist} we have, for a.e. $t\geq 0$ (including $t=0$),
\begin{align}
&
\int_t^{t+1}\Big(\frac{\nu_1}{2}\Vert\nabla u\Vert^2+\Vert\nabla\mu\Vert^2\Big)d\tau
\leq\mathcal{E}(z(t))-\mathcal{E}(z(t+1))+\frac{1}{2\nu_1}\int_t^{t+1}\|h(\tau)\|_{V_{div}'}^2d\tau.
\label{q2}
\end{align}
Furthermore, by means of (A2) and by multiplying the gradient of \eqref{chempot}
by $\nabla\varphi$, it can be shown that (see \cite{CFG})
\begin{equation}
\|\nabla\mu\|^2\geq k_3\|\nabla\varphi\|^2-k_4\|\varphi\|^2,
\label{q3}
\end{equation}
where $k_3=c_0^2/4$ and $k_4=2\|\nabla J\|_{L^1}^2$.
Therefore, combining \eqref{q2} and \eqref{q3} with \eqref{contrE} and \eqref{q1} we get
\begin{equation}
\int_t^{t+1}\Big(\frac{\nu_1}{2}\Vert\nabla u\Vert^2+k_3\Vert\nabla\varphi\Vert^2\Big)d\tau
\leq c(\|\varphi\|_{L^{\infty}}^{p'}+\|u\|_{L^{\infty}}^2)e^{-kt}
+cK+c,
\label{eniqtt+1}
\end{equation}
from which we deduce
\begin{equation}
\|T(t)\varphi\|_{L^2_{tb}(0,\infty;V)}^2+\|T(t)u\|_{L^2_{tb}(0,\infty;V_{div})}^2
\leq c(\|\varphi\|_{L^{\infty}}^{p'}+\|u\|_{L^{\infty}}^2)e^{-kt}
+cK+c,
\label{2term}
\end{equation}
for all $t\geq 1$.
Let us come to the terms in \eqref{trajest} containing the time derivatives. As far as the contribution
arising from the Korteweg force term is concerned, since, if $d=3$ we have
\begin{align*}
&
\|\varphi\nabla\mu\|_{V_{div}'}\leq c\|\varphi\|_{L^3(\Omega)}\|\nabla\mu\|
\leq c\Vert \varphi\Vert^{1/2}\Vert\varphi\Vert_{L^6(\Omega)}^{1/2}\Vert\nabla\mu\Vert
\leq c\Vert \varphi\Vert^{1/2}\Vert\varphi\Vert_V^{1/2}\Vert\nabla\mu\Vert,
\end{align*}
then, on account of \eqref{q1}, of the fact that $\|\varphi\|\leq\|\varphi\|_{L^{p'}(\Omega)}^{p'/2}+c$,
and of \eqref{q2}, \eqref{eniqtt+1} we get
\begin{align}
&\Big (\int_t^{t+1}\|\varphi\nabla\mu\|_{V_{div}'}^{4/3}d\tau\Big)^{3/4}
\leq c\|\varphi\|_{L^{\infty}(t,t+1;H)}^{1/2}\Big(\int_t^{t+1}\|\varphi\|_V^2d\tau\Big)^{1/4}
\Big(\int_t^{t+1}\|\nabla\mu\|^2d\tau\Big)^{1/2}\nonumber\\
&\leq\Big(c(\|\varphi\|_{L^{\infty}}^{p'}+\|u\|_{L^{\infty}}^2)e^{-kt}
+K+c\Big)^{1/4}\Big(c(\|\varphi\|_{L^{\infty}}^{p'}+\|u\|_{L^{\infty}}^2)e^{-kt}
+cK+c\Big)^{1/4}\nonumber\\
&
\Big(c(\|\varphi\|_{L^{\infty}}^{p'}+\|u\|_{L^{\infty}}^2)e^{-kt}
+cK+c\Big)^{1/2}\leq c(\|\varphi\|_{L^{\infty}}^{p'}+\|u\|_{L^{\infty}}^2)e^{-kt}
+cK+c,
\label{kort1}
\end{align}
for all $t\geq 1$.
If $d=2$ then we have
\begin{align*}
&
\|\varphi\nabla\mu\|_{V_{div}'}\leq c\Vert\varphi\Vert_{L^{2+2q}(\Omega)}\Vert\nabla\mu\Vert
= c\Vert\varphi\Vert_{L^{p'}(\Omega)}\Vert\nabla\mu\Vert.
\end{align*}
Therefore, recalling \eqref{q1}, \eqref{eniqtt+1} and \eqref{q2}, we obtain
\begin{align}
&\Big (\int_t^{t+1}\|\varphi\nabla\mu\|_{V_{div}'}^2 d\tau\Big)^{1/2}
\leq c\|\varphi\|_{L^{\infty}(t,t+1;L^{p'}(\Omega))}\Big(\int_t^{t+1}\|\nabla\mu\|^2 d\tau\Big)^{1/2}\nonumber\\
&\leq\Big(c(\|\varphi\|_{L^{\infty}}^{p'}+\|u\|_{L^{\infty}}^2)e^{-kt}
+K+c\Big)^{1/2}\Big(c(\|\varphi\|_{L^{\infty}}^{p'}+\|u\|_{L^{\infty}}^2)e^{-kt}
+cK+c\Big)^{1/2}\nonumber\\
&\leq c(\|\varphi\|_{L^{\infty}}^{p'}+\|u\|_{L^{\infty}}^2)e^{-kt}
+cK+c,\qquad\forall t\geq 1.
\label{kort2}
\end{align}
Therefore \eqref{kort1} and \eqref{kort2} for $d=3,2$ entail
\begin{align}
&\|T(t)(\varphi\nabla\mu)\|_{L^{4/d}_{tb}(0,\infty;V_{div}')}\leq
c(\|\varphi\|_{L^{\infty}}^{p'}+\|u\|_{L^{\infty}}^2)e^{-kt}
+cK+c,\qquad\forall t\geq 1.
\label{uttb1}
\end{align}
Furthermore, for $d=3$, recalling \eqref{standest3D} we have
\begin{align*}
&
\Big(\int_t^{t+1}\|\mathcal{B}(u,u)\|_{V_{div}'}^{4/3}d\tau\Big)^{3/4}
\leq  c\|u\|_{L^{\infty}(t,t+1;G_{div})}^{1/2}\Big(\int_t^{t+1}\|\nabla u\|^2d\tau\Big)^{3/4}\nonumber\\
&\leq\Big(c(\|\varphi\|_{L^{\infty}}^{p'}+\|u\|_{L^{\infty}}^2)e^{-kt}
+K+c\Big)^{1/4}\Big(c(\|\varphi\|_{L^{\infty}}^{p'}+\|u\|_{L^{\infty}}^2)e^{-kt}
+cK+c\Big)^{3/4}\nonumber\\
&\leq c(\|\varphi\|_{L^{\infty}}^{p'}+\|u\|_{L^{\infty}}^2)e^{-kt}
+cK+c,\qquad\forall t\geq 1,
\end{align*}
while, for $d=2$, recalling \eqref{standest2D} we obtain
\begin{align*}
&
\Big(\int_t^{t+1}\|\mathcal{B}(u,u)\|_{V_{div}'}^2 d\tau\Big)^{1/2}
\leq  c\|u\|_{L^{\infty}(t,t+1;G_{div})}\Big(\int_t^{t+1}\|\nabla u\|^2d\tau\Big)^{1/2}\nonumber\\
&\leq c(\|\varphi\|_{L^{\infty}}^{p'}+\|u\|_{L^{\infty}}^2)e^{-kt}
+cK+c,\qquad\forall t\geq 1.
\end{align*}
Hence, for $d=3,2$ we get
\begin{align}
&\|T(t)\mathcal{B}(u,u)\|_{L^{4/d}_{tb}(0,\infty;V_{div}')}
\leq c(\|\varphi\|_{L^{\infty}}^{p'}+\|u\|_{L^{\infty}}^2)e^{-kt}
+cK+c,\qquad\forall t\geq 1.
\label{uttb2}
\end{align}
Recalling equation \eqref{sy3} which can be written as
$$
u_t+\mathcal{A}(u,\varphi)+\mathcal{B}(u,u)=-\varphi\nabla\mu+h\qquad\mbox{in }V_{div}',
\;\mbox{a.e. in } (0,\infty),
$$
we deduce by comparison that
\begin{align*}
\|u_t\|_{L^{4/d}(t,t+1;V_{div}')}&\leq\nu_2\|u\|_{L^2(t,t+1;V_{div})}
+\|\mathcal{B}(u,u)\|_{L^{4/d}(t,t+1;V_{div}')}\\
&+\|\varphi\nabla\mu\|_{L^{4/d}(t,t+1;V_{div}')}+\|h\|_{L^2(t,t+1;V_{div}')}.
\end{align*}
Therefore, using \eqref{uttb1} and \eqref{uttb2}, we obtain
\begin{align}
&
\|T(t)u_t\|_{L^{4/d}_{tb}(0,\infty;V_{div}')}\leq
 c(\|\varphi\|_{L^{\infty}}^{p'}+\|u\|_{L^{\infty}}^2)e^{-kt}
 +cK+c,\qquad\forall t\geq 1.
 \label{3term}
\end{align}
Now, from \eqref{sy1}, for $d=3$ we can write
$$\|\varphi_t\|_{V'}\leq \|\nabla\mu\|+c\|u\|_{L^{2(1+\frac1q)}(\Omega)}\|\varphi\|_{L^{2+2q}(\Omega)}
\leq \|\nabla\mu\|+c\|\nabla u\|\|\varphi\|_{L^{p'}(\Omega)},$$
while, for $d=2$ we have
$$\|\varphi_t\|_{V'}\leq \|\nabla\mu\|+c\|\nabla u\|\|\varphi\|_{L^{2+2q}(\Omega)}
= \|\nabla\mu\|+c\|\nabla u\|\|\varphi\|_{L^{p'}(\Omega)}.$$
The contribution from the transport term gives
\begin{align*}
&\Big(\int_t^{t+1}\|\nabla u\|^2\|\varphi\|_{L^{p'}(\Omega)}^2 d\tau\Big)^{1/2}
\leq\|\varphi\|_{L^{\infty}(t,t+1;L^{p'}(\Omega))}\Big(\int_t^{t+1}\|\nabla u\|^2d\tau\Big)^{1/2}
\nonumber\\
&\leq
 c(\|\varphi\|_{L^{\infty}}^{p'}+\|u\|_{L^{\infty}}^2)e^{-kt}
 +cK+c,\qquad\forall t\geq 1.
\end{align*}
Thus, in both cases $d=2$ and $d=3$, we find
\begin{align}
&\|T(t)\varphi_t\|_{L^2_{tb}(0,\infty;V')}\leq
c(\|\varphi\|_{L^{\infty}}^{p'}+\|u\|_{L^{\infty}}^2)e^{-kt}
 +cK+c,\qquad\forall t\geq 1.
 \label{4term}
\end{align}
Finally, collecting \eqref{1term}, \eqref{2term}, \eqref{3term} and \eqref{4term},
we get \eqref{trajest} with $\Lambda_0=c$ and $\Lambda_1=cK+c$.
\end{proof}

Propositions \ref{traj1} and \ref{traj2} are the basic ingredients to establish next theorem,
which is the main result
of this section. We denote by $Z(h_0)$ the set of all complete symbols in $\mathcal{H}_+(h_0)$,
i.e., the set of functions $\zeta:\mathbb{R}\to V_{div}'$, $\zeta\in L^2_{loc}(\mathbb{R};V_{div}')$
such that $\Pi_+T(t)\zeta\in\omega(\mathcal{H}_+(h_0))$, for all $t\in\mathbb{R}$,
where $\Pi_+$ is the restriction operator to the semiaxis $[0,\infty)$.
To any complete symbol $\zeta\in Z(h_0)$ there corresponds, by \cite[Chap XIV, Definition 2.5]{CV}
(see also \cite[Definition 4.4]{CV2}), the kernel $\mathcal{K}_{\zeta}$ which consists of all weak solutions $z:\mathbb{R}\to G_{div}\times H$
with external force $\zeta$ (in the sense of Definition \ref{wfdfn2} with $T\in\mathbb{R}$)
satisfying inequality \eqref{eist} on $\mathbb{R}$ and that are bounded in the space $\mathcal{F}_b$
(the space $\mathcal{F}_b$ is defined as $\mathcal{F}_b^+$ with the time semiaxis $[0,\infty)$ replaced with $\mathbb{R}$ in the definition of $\mathcal{F}_b^+$; in the same way $\mathcal{F}_{loc}$ and $\Theta_{loc}$
can be defined). Then, we set
$$\mathcal{K}_{Z(h_0)}:=\bigcup_{\zeta\in Z(h_0)}\mathcal{K}_{\zeta}.$$

\begin{thm}
\label{traj3} Let (A1)-(A3) hold. In addition, suppose that (A5) holds with  $p\in (\frac65,\frac32]$ if $d=3$ and
with $p\in (1,2)$ if $d=2$ and that
(A6) holds with $2q+2=p^\prime$. If
\begin{align*}
& h_0 \in L^2_{tb}(0,\infty;V^\prime_{div}), \quad d=2,\\
& h_0 \in L^2_{tb}(0,\infty;G_{div}), \quad d=3,
\end{align*}
then $\{T(t)\}$ acting on $\mathcal{K}^+_{\mathcal{H}(h_0)}$
possesses the uniform (with respect to $h\in \mathcal{H}(h_0)$)
trajectory attractor $\mathcal{U}_{\mathcal{H}(h_0)}$. This set is
bounded in $\mathcal{F}^+_{b}$ and compact in
$\Theta^+_{loc}$. Moreover , we have
$$\mathcal{U}_{\mathcal{H}(h_0)}=\mathcal{U}_{\omega(\mathcal{H}_{+}(h_0))}=\mathcal{K}_{Z(h_0)},$$
where $\mathcal{U}_{\omega(\mathcal{H}_{+}(h_0))}$ is the uniform (with respect to $h\in\omega(\mathcal{H}_{+}(h_0))$)
trajectory attractor of the family $\{\mathcal{K}_h^+:h\in\omega(\mathcal{H}_{+}(h_0))\}$,
$\mathcal{U}_{\omega(\mathcal{H}_{+}(h_0))}\subset\mathcal{K}_{\omega(\mathcal{H}_{+}(h_0)}^+$.
The kernel $\mathcal{K}_{\zeta}$ is not empty for any $\zeta\in Z(h_0)$; the set $\mathcal{K}_{Z(h_0)}$
is bounded in $\mathcal{F}_b$ and compact in $\Theta_{loc}$.
\end{thm}
\begin{proof}
The family of trajectory spaces
$\{\mathcal{K}_h^+:h\in\mathcal{H}_+(h_0)\}$ is $(\Theta_{loc}^+,\mathcal{H}_+(h_0))-$closed
due to Proposition \ref{traj1}.
Notice that the assumption on $h_0$ ensures that the symbol space $\Sigma:=\mathcal{H}_+(h_0)$
is a compact metric space.
Thanks to \eqref{trajest} it is easy to see that the ball
$$B_{\mathcal{F}_b^+}(2\Lambda_1):=\{[v,\psi]\in\mathcal{F}_b^+:\Vert[v,\psi]\Vert_{\mathcal{F}_b^+}\leq 2\Lambda_1\}$$
is a uniformly (w.r.t. $h\in\mathcal{H}_+(h_0)$) absorbing set for the family
$\{\mathcal{K}_h^+:h\in\mathcal{H}_+(h_0)\}$. The ball $B_{\mathcal{F}_b^+}(2\Lambda_1)$
is compact in $\Theta_{loc}^+$ and bounded in $\mathcal{F}_b^+$.
The conditions of \cite[Chap.~XIV, Thm 2.1 and Thm.~3.1]{CV}) are thus satisfied and the thesis follows.
\end{proof}

\bigskip

\noindent {\bf Acknowledgments.} This work was partially supported
by the Italian MIUR-PRIN Research Project 2008 ``Transizioni di fase,
isteresi e scale multiple''. The first author was also supported by the
FTP7-IDEAS-ERC-StG Grant $\sharp$200497(BioSMA) and
the FP7-IDEAS-ERC-StG Grant \#256872 (EntroPhase).


\begin{thebibliography}{50}


\bibitem{A1} H. Abels, {\itshape On a diffusive interface model for
    two-phase flows of viscous, incompressible fluids with matched
    densities}, Arch. Ration. Mech. Anal.  \textbf{194} (2009),
    463-506.

\bibitem{A2} H. Abels, {\itshape Existence of weak solutions for a
    diffuse
    interface model for viscous, incompressible fluids with general
    densities}, Comm. Math. Phys. \textbf{289} (2009), 45-73.

\bibitem{A3} H. Abels, {\itshape Longtime behavior of
    solutions of a Navier-Stokes/Cahn-Hilliard system}, Proceedings of the
    Conference ``Nonlocal and Abstract Parabolic Equations and their
    Applications'', Bedlewo, Banach Center Publ. \textbf{86} (2009), 9-19.

\bibitem{AF} H. Abels, E. Feireisl, {\itshape On a diffuse interface
    model for a two-phase flow of compressible viscous fluids}, Indiana
    Univ. Math. J. \textbf{57} (2008), 659-698.


\bibitem{AMW} D.M. Anderson, G.B. McFadden, A.A. Wheeler,
    {\itshape Diffuse-interface methods in fluid mechanics},
    Annu. Rev. Fluid Mech. \textbf{30}, Annual Reviews, Palo
    Alto, CA, 1998, 139-165.

\bibitem{BCB} V.E. Badalassi, H.  Ceniceros, S. Banerjee,
    {\itshape Computation of multiphase systems with phase field models},
J. Comput. Phys. \textbf{190} (2003), 371-397.

\bibitem{Ba} J.M. Ball, {\itshape Continuity properties and global
    attractors of generalized semiflows and the Navier-Stokes
    equation}, J. Nonlinear Sci. \textbf{7} (1997), 475-502 (Erratum, J.
    Nonlinear Sci. \textbf{8} (1998), 233).

\bibitem{Ba2} J. M. Ball, {\itshape Global attractors for damped
    semilinear wave equations}, Discrete Contin. Dyn. Syst. \textbf{10} (2004)
    31-52.





\bibitem{BH1} P.W. Bates, J. Han, {\itshape The Neumann
    boundary problem for a nonlocal Cahn-Hilliard equation}, J.
    Differential Equations \textbf{212} (2005), 235-277.

\bibitem{BH2} P.W. Bates, J. Han, {\itshape The Dirichlet
    boundary problem for a nonlocal Cahn-Hilliard equation}, J.
    Math. Anal. Appl. \textbf{311} (2005), 289-312.


\bibitem{B} F. Boyer, {\itshape Mathematical study of multi-phase
    flow under shear through order parameter formulation}, Asymptot.
    Anal. \textbf{20} (1999), 175-212.

\bibitem{B2} F. Boyer, {\it Nonhomogeneous Cahn-Hilliard fluids},
 Ann. Inst. H. Poincar\'e Anal. Non Lin\'eaire   \textbf{18} (2001),
 225-259.

\bibitem{B3} F. Boyer, {\it A theoretical and numerical model for
    the study of incompressible mixture flows},
    Comput. \& Fluids \textbf{31} (2002), 41-68.







\bibitem{CV} V.V. Chepyzhov, M. Vishik, {\itshape Attractors for
    Equations of Mathematical Physics}, Amer. Math. Soc. Colloq. Publ.,
    vol. \textbf{49}, American Mathematical Society, Providence, RI, 2002.

\bibitem{CV2} V.V. Chepyzhov, M. Vishik, {\itshape Evolution equations and their trajectory attractors},
    J. Math. Pures Appl. \textbf{76} (1997), 913-964.

\bibitem{CF} A. Cheskidov, C. Foias, {\itshape On global attractors of the 3D-Navier-Stokes equations},
J. Differential Equations \textbf{231} (2006), 714-754.

\bibitem{CFG} P. Colli, S. Frigeri, M. Grasselli, {\it Global
    existence of weak solutions to a nonlocal
    Cahn-Hilliard-Navier-Stokes system}, submitted.

\bibitem{Cu} N.J. Cutland, {\itshape Global attractors for small samples and germs of 3D
Navier-Stokes equations}, Nonlinear Anal. \textbf{62} (2005), 265-281.

\bibitem{D} M. Doi, {\itshape Dynamics of domains and textures},
    Theoretical Challenges in the Dynamics of Complex Fluids (T.C. McLeish Ed.),
    NATO-ASI Ser. \textbf{339},  Kluwer Academic, Dordrecht, 1997, 293-314.


\bibitem{F} X. Feng, {\itshape Fully discrete finite element
    approximation of the Navier-Stokes-Cahn-Hilliard diffuse interface
    model for two-phase flows}, SIAM J. Numer. Anal.
    \textbf{44} (2006), 1049-1072.

\bibitem{FS} F. Flandoli, B. Schmalfuss, {\itshape Weak solutions and attractors
for the 3-dimensional Navier-Stokes equations with nonregular force,} J.
Dynam. Differential Equations \textbf{11} (1999), 355-398.


\bibitem{G} H. Gajewski, {\itshape On a nonlocal model of
    non-isothermal phase separation}, Adv. Math. Sci. Appl.
    \textbf{12} (2002), 569-586.

\bibitem{GZ} H. Gajewski, K. Zacharias, {\itshape On a nonlocal
    phase separation model}, J. Math. Anal. Appl. \textbf{286} (2003),
    11-31.

\bibitem{GG1} C.G. Gal, M. Grasselli, {\itshape Asymptotic
    behavior of
    a Cahn-Hilliard-Navier-Stokes system in 2D},  Ann.
Inst. H. Poincar\'e Anal. Non Lin\'eaire \textbf{27} (2010), 401-436.

\bibitem{GG2} C.G. Gal, M. Grasselli, {\itshape Trajectory attractors
    for binary fluid mixtures in 3D}, Chinese Ann. Math. Ser. B
    \textbf{31} (2010), 655-678.

\bibitem{GG3} C.G. Gal, M. Grasselli, {\itshape Instability of
    two-phase
    flows: a lower bound on the dimension of the global attractor of the
    Cahn-Hilliard-Navier-Stokes system,} Phys. D \textbf{240} (2011),
    629-635.

\bibitem{GP} M. Grasselli, D. Pra\v z\'ak, {\itshape Longtime behavior of a
diffuse interface model for binary fluid mixtures with shear dependent viscosity,}
submitted.


\bibitem{GL1} G. Giacomin, J.L. Lebowitz, {\itshape Phase
    segregation dynamics in particle systems with long range
    interactions. I. Macroscopic limits}, J. Statist. Phys. \textbf{87}
    (1997), 37-61.

\bibitem{GL2} G. Giacomin, J.L. Lebowitz, {\itshape Phase
    segregation dynamics in particle systems with long range
    interactions. II. Phase motion}, SIAM J. Appl. Math.  \textbf{58}
    (1998), 1707-1729.

\bibitem{GPV} M.E. Gurtin, D. Polignone, J. Vi\~{n}als,
    {\itshape Two-phase binary fluids and immiscible fluids described by
    an order parameter}, Math. Models Meth. Appl. Sci. \textbf{6} (1996),
    8-15.

\bibitem{H} J. Han, {\itshape The Cauchy problem and steady state
    solutions for a nonlocal Cahn-Hilliard equation}, Electron. J. Differential Equations
    \textbf{113} (2004), 9 pp.

\bibitem{Has} B. Haspot, {\itshape Existence of global weak solution for compressible fluid
models with a capillary tensor for discontinuous interfaces},
Differential Integral Equations \textbf{23} (2010), 899-934.

\bibitem{HH} P.C. Hohenberg, B.I. Halperin, {\itshape Theory
    of dynamical critical phenomena}, Rev. Mod. Phys.
    \textbf{49} (1977), 435-479.

\bibitem{JV} D. Jasnow, J. Vi\~{n}als, {\itshape Coarse-grained description of thermo-capillary flow},
Phys. Fluids \textbf{8} (1996), 660-669.

\bibitem{KaVa} A.V. Kapustyan, J. Valero, {\itshape Weak and strong attractors for the
3D Navier-Stokes system},
J. Differential Equations \textbf{240} (2007), 249-278.

\bibitem{KSW} D. Kay, V. Styles, R. Welford, {\itshape Finite
    element approximation of a Cahn-Hilliard-Navier-Stokes system},
    Interfaces Free Bound. \textbf{10} (2008), 5-43.

\bibitem{KKL} J. Kim, K. Kang, J. Lowengrub,
    {\itshape Conservative multigrid methods for Cahn-Hilliard fluids},
    J. Comput. Phys. \textbf{193} (2004), 511-543.

\bibitem{KlVa} P.E. Kloeden, J. Valero, {\itshape The Kneser property of the weak solutions of the three dimensional Navier-Stokes equations}, Discrete Contin. Dyn. Syst. \textbf{28} (2010), 161-179.

\bibitem{LS} C. Liu, J.  Shen, {\itshape A phase field model for the
    mixture of two incompressible fluids and its approximation by a
    Fourier-spectral method}, Phys. D \textbf{179} (2003), 211-228.





\bibitem{LP} S.-O. Londen, H. Petzeltov\'{a}, {\itshape Convergence
    of solutions of a non-local phase-field system}, Discrete Contin. Dyn.
Syst. Ser. S \textbf{4} (2011), 653-670.

\bibitem{LT} J. Lowengrub, L. Truskinovsky, {\itshape
    Quasi-incompressible Cahn-Hilliard fluids and
    topological transitions}, Proc. R. Soc. London A \textbf{454} (1998),
    2617-2654.

\bibitem{MRR} P. Mar\'{\i}n-Rubio, J. Real, {\itshape Pullback
    attractors for 2D-Navier-Stokes equations with delays in continuous and
 sub-linear operators}, Discrete Contin. Dyn. Syst. \textbf{26} (2010),
 989-1006.


\bibitem{M} A. Morro, {\itshape Phase-field models of Cahn-Hilliard
    Fluids and extra fluxes}, Adv. Theor. Appl. Mech. \textbf{3} (2010),
    409-424.

\bibitem{R} C. Rhode, {\itshape On local and non-local
    Navier-Stokes-Korteweg systems for liquid-vapour phase transitions},
    Z. Angew. Math. Mech. \textbf{85} (2005), 839-857.

\bibitem{R2} C. Rhode, {\itshape A Local and Low-Order Navier-Stokes-Korteweg System}, Nonlinear partial differential equations and hyperbolic wave phenomena (H. Holden and K.H. Karlsen Eds.), 315-337, Contemp. Math., \textbf{526}, Amer. Math. Soc., Providence, RI, 2010.

\bibitem{Ro} R.M.S. Rosa, {\itshape Asymptotic regularity conditions for the strong
convergence towards weak limit sets and weak attractors
of the 3D Navier-Stokes equations}, J. Differential Equations \textbf{229} (2006), 257-269.

\bibitem{Se} G.R. Sell, {\itshape Global attractors for the three-dimensional
Navier-Stokes equations,} J. Dynam. Differential Equations \textbf{8} (1996), 1-33.

\bibitem{SY} J. Shen, X. Yiang, {\itshape Energy stable schemes for
    Cahn-Hilliard phase-field model of two-phase incompressible flows},
    Chinese Ann. Math. Ser. B \textbf{31} (2010), 743-758.



\bibitem{S} V.N. Starovoitov, {\itshape The dynamics of a
    two-component fluid in the presence of capillary forces},
    Math. Notes \textbf{62} (1997), 244-254.


\bibitem{T} R. Temam, {\itshape Navier-Stokes equations and
    nonlinear functional analysis}, Second edition, CBMS-NSF Reg. Conf.
    Ser. Appl. Math., \textbf{66}, SIAM, Philadelphia, PA, 1995.

\bibitem{ZWH} L. Zhao, H. Wu, H. Huang, {\itshape Convergence to
    equilibrium for a phase-field model for the
    mixture of two viscous incompressible fluids},
    Commun. Math. Sci. \textbf{7} (2009), 939-962.





\end{thebibliography}
\end{document}